\documentclass[a4paper, 11pt]{amsart}

\usepackage{a4wide}
\usepackage[english]{babel} 
\usepackage{amsmath, amssymb} 

\usepackage{tikz}
\usepackage{tikz-cd}
	\tikzcdset{row sep/normal=1.3em}
	\tikzcdset{column sep/normal=1.3em}

\usepackage[mathscr]{euscript}

\usepackage{hyperref}
\usepackage[capitalize,noabbrev]{cleveref}

\newcommand{\Sch}{\mathrm{Sch}}

\newcommand{\Q}{\mathbb{Q}}

\newcommand{\N}{\mathbb{N}}
\newcommand{\Z}{\mathbb{Z}}

\newcommand{\T}{\mathbb{T}}
\newcommand{\M}{\mathrm{M}}

\newcommand{\KH}{{KH}}
\newcommand{\cZ}{\mathcal{Z}}

\newcommand{\Y}{\mathscr{Y}}
\newcommand{\X}{\mathscr{X}}
\newcommand{\D}{\mathscr{D}}

\newcommand{\cO}{\mathcal{O}}
\newcommand{\cA}{\mathcal{A}}
\newcommand{\cB}{\mathcal{B}}

\newcommand{\cT}{\mathcal{T}}

\newcommand{\cC}{\mathcal{C}}
\newcommand{\cD}{\mathcal{D}}
\newcommand{\cE}{\mathcal{E}}

\newcommand{\cn}{\mathrm{cn}}

\DeclareMathOperator{\Pro}{Pro}
\newcommand{\Alg}{\mathrm{Alg}}
\newcommand{\QCoh}{\mathrm{QCoh}}

\newcommand{\cQ}{\mathcal{Q}}
\newcommand{\cR}{\mathcal{R}}

\newcommand{\bc}{\vec{c}}

\newcommand{\E}{\mathbb{E}}

\newcommand{\can}{\mathrm{can}}

\newcommand{\swedge}{{\scriptscriptstyle\wedge}}

\newcommand{\op}{\mathrm{op}}
\newcommand{\Cat}{\mathrm{Cat}}
\renewcommand{\Pr}{\mathrm{Pr}}

\newcommand{\perf}{\mathrm{perf}}

\newcommand{\inv}{\mathrm{inv}}
\renewcommand{\inf}{\mathrm{inf}}
\newcommand{\GL}{\mathrm{GL}}
\newcommand{\BGL}{\mathrm{BGL}}
\newcommand{\CAlg}{\mathrm{CAlg}}

\newcommand{\lto}{\longrightarrow}

\newcommand{\mmod}{/\!\!/}

\newcommand{\id}{\mathrm{id}}
\DeclareMathOperator{\Bl}{Bl}

\newcommand{\pr}{\mathrm{pr}}

\DeclareMathOperator{\Spec}{Spec}
\DeclareMathOperator{\Sp}{Sp}

\DeclareMathOperator*{\colim}{colim}
\newcommand{\laxtimes}[1]{\mathop{\times\mkern-13mu\raise1.3ex\hbox{$\scriptscriptstyle\to$}_{#1}}}
\newcommand{\slax}{ \times\mkern-14mu\raise1ex\hbox{$\scriptscriptstyle\to$} }

\DeclareMathOperator{\RMod}{RMod}
\DeclareMathOperator{\Mod}{Mod}
\DeclareMathOperator{\Perf}{Perf}
\DeclareMathOperator{\Tor}{Tor}

\DeclareMathOperator{\map}{map}
\DeclareMathOperator{\End}{End}
\DeclareMathOperator{\Fun}{Fun}

\DeclareMathOperator{\Ind}{Ind}
\DeclareMathOperator{\fib}{fib}
\DeclareMathOperator{\Idem}{Idem}
\DeclareMathOperator{\cof}{cofib}
\DeclareMathOperator{\TC}{TC}

\DeclareMathOperator{\THH}{THH}
\DeclareMathOperator{\HC}{HC}
\DeclareMathOperator{\HH}{HH}
\DeclareMathOperator{\HN}{HN}
\DeclareMathOperator{\HP}{HP}

\newcommand{\ol}[1]{\overline{#1}}

\newcommand{\wtimes}[2]{\odot_{#1}^{#2}}

\newtheorem{thm}{Theorem}
\newtheorem*{thm*}{Theorem}
\newtheorem{cor}[thm]{Corollary}
\newtheorem*{cor*}{Corollary}
\newtheorem{lemma}[thm]{Lemma}
\newtheorem{prop}[thm]{Proposition}
\newtheorem*{mainthm}{Main Theorem}

\theoremstyle{definition}
\newtheorem{dfn}[thm]{Definition}
\newtheorem*{dfn*}{Definition}
\newtheorem{Notation}[thm]{Notation}

\theoremstyle{remark}

\numberwithin{thm}{section}

\newtheorem{rem}[thm]{Remark}
\newtheorem*{rem*}{Remark}

\newtheorem*{ex*}{Example}
\newtheorem{ex}[thm]{Example}

\newtheorem*{step1}{Step 1}
\newtheorem*{step2}{Step 2}
\newtheorem*{step3}{Step 3}
\newtheorem*{step4}{Step 4}

\theoremstyle{plain}
\newcounter{zaehler}

\newtheorem{introthm}[zaehler]{Theorem}

\title{On the \textit{K}-theory of pullbacks}
\author{Markus Land}
\address{Fakult\"at f\"ur Mathematik, Universit\"at Regensburg, 93040 Regensburg, Germany}
\email{markus.land@ur.de}
\author{Georg Tamme}
\address{Fakult\"at f\"ur Mathematik, Universit\"at Regensburg, 93040 Regensburg, Germany}
\email{georg.tamme@ur.de}

\thanks{The authors are supported by the CRC/SFB 1085 \emph{Higher Invariants} (Universit\"at Regensburg) funded by the DFG}

\date{\today}

\begin{document}

\begin{abstract}
To any pullback square of ring spectra we associate a new ring spectrum and use it to describe the failure of excision in algebraic $K$-theory.
The construction of this new ring spectrum
is categorical and hence allows to determine the failure of excision for any localizing invariant in place of $K$-theory.  

As immediate consequences we obtain an improved version of Suslin's excision result in $K$-theory, generalizations of results of Geisser and Hesselholt on torsion in (bi)relative $K$-groups, and a generalized version of pro-excision for $K$-theory. Furthermore, we show that any truncating invariant satisfies excision, nilinvariance, and cdh-descent. Examples of truncating invariants include the fibre of the cyclotomic trace, the fibre of the rational Goodwillie--Jones Chern character, 
periodic cyclic homology in characteristic zero, and homotopy $K$-theory.

Various of the results we obtain have been known previously, though most of them in weaker forms and with less direct proofs.
\end{abstract}

\maketitle
\setcounter{tocdepth}{2}

\section*{Introduction}

The classical excision theorem in algebraic $K$-theory due to Milnor, Bass, and Murthy says that any Milnor square of rings
\begin{equation}
	\label{diag:ring-spectra}\tag{$\square$}
		\begin{tikzcd}
		A \ar[r] \ar[d] & B \ar[d] \\ 
		A' \ar[r] & B' 
		\end{tikzcd}
\end{equation}
i.e.~a pullback square of rings with  $B \to B'$ surjective, gives rise to a long exact sequence of  $K$-groups (the excision sequence)
\begin{equation}
	\label{eq:Bass-Milnor-sequence}
	K_1(A) \lto K_1(A')\oplus K_1(B) \lto K_1(B') \lto K_0(A) \lto K_0(A')\oplus K_0(B) \lto \dots
\end{equation}
starting at $K_{1}(A)$, see \cite[Theorem XII.8.3]{Bass}. This sequence is extremely useful for computations of  negative $K$-groups. 
Unfortunately, for a long time it seemed to be impossible to extend this sequence to the left in a natural way, which made computations of higher $K$-groups much harder.
For instance, Swan proved that there is no functor $\text{`$K_{2}$'}$ from rings to abelian groups such that the exact sequence \eqref{eq:Bass-Milnor-sequence} could be extended to an exact sequence
\[
\text{`$K_{2}$'}(A') \oplus \text{`$K_{2}$'}(B) \lto \text{`$K_{2}$'}(B') \lto K_{1}(A) \lto \dots
\]    
see \cite[Corollary 1.2]{Swan}. 
Our main insight is that to any Milnor square \eqref{diag:ring-spectra} one can in fact functorially associate a connective ring spectrum $A' \wtimes{A}{B'} B$  together with natural maps from $A'$ and $B$ and  to $B'$ for which there is a long exact sequence
\begin{equation}
	\label{eq:long-boxtimes-sequence}
\dots \lto  K_{i}(A) \lto K_i(A')\oplus K_i(B) \lto K_i(A' \wtimes{A}{B'} B) \lto K_{i-1}(A) \lto \dots \quad (i\in \Z)
\end{equation}
This sequence coincides with the sequence \eqref{eq:Bass-Milnor-sequence} in degrees $\leq 1$.

In fact, our main result is more general. Before formulating it, we remark two things. Firstly, a Milnor square is a particular example of a pullback square of $\E_{1}$-ring spectra. Secondly, non-connective algebraic $K$-theory is a particular example of a  localizing invariant. By the latter we mean a functor   from the $\infty$-category of small, stable, idempotent complete  $\infty$-categories to some stable $\infty$-category  that takes exact sequences to fibre sequences.  A localizing invariant $E$ gives rise to a functor on $\E_{1}$-ring spectra by evaluating $E$ on the $\infty$-category of perfect  modules: $E(A) = E(\Perf(A))$. 
Our main result is the following.

\begin{mainthm}
	\label{intro-main-result}
Assume that \eqref{diag:ring-spectra} is a pullback square of $\E_1$-ring spectra. Associated to this square there exists a natural $\E_1$-ring spectrum $A' \wtimes{A}{B'} B$ with the following properties: The original diagram \eqref{diag:ring-spectra} extends to  a  commutative diagram of $\E_{1}$-ring spectra
\begin{equation}
	\label{diag:natural-diagram0}
	\begin{tikzcd}
	 A \ar[d]\ar[r] & B \ar[d] \ar[ddr, bend left=20] & \\ 
	 A' \ar[r] \ar[drr, bend right=15] & A' \wtimes{A}{B'} B \ar[dr] & \\
	 & & B'
	\end{tikzcd}
\end{equation}
and any localizing invariant sends the inner square in~\eqref{diag:natural-diagram0} to a pullback square. 

The underlying spectrum of $A' \wtimes{A}{B'} B$ is equivalent to $A'\otimes_A B$, and the underlying diagram of spectra is the canonical one.
\end{mainthm}

 \begin{rem*}
If \eqref{diag:ring-spectra} is a diagram of $\E_k$-ring spectra, then the derived tensor product $A' \otimes_{A} B$ carries a natural $\E_{k-1}$-structure. Hence, if $k \geq 2$ then $A' \otimes_{A} B$ carries a natural  $\E_{1}$-structure, which in general  differs from the one on $A' \wtimes{A}{B'} B$. We work out an illuminating example in Section~\ref{sec:example}. 
\end{rem*}

As mentioned earlier, the absence of an excision long exact sequence of $K$-groups in positive degrees makes computations for positive $K$-groups  difficult. To remedy this situation, 
there have been essentially \emph{three approaches}:

\medskip

The \textit{first approach} is due to
 Suslin and Wodzicki \cite{Suslin-Wod,Suslin}.  Suslin proves \cite[Theorem A]{Suslin}  that for a Milnor square the sequence \eqref{eq:Bass-Milnor-sequence} in fact does extend  up to degree~$n$ provided the Tor groups $\Tor^{\Z\ltimes I}_{i}(\Z,\Z)$ vanish  for $i=1,\dots,n-1$. Here  $I$ denotes the kernel of the map $A \to A'$, and $\Z \ltimes I$ its unitalization. 
As a simple consequence of 
our Main Theorem we get the following generalization and strengthening of Suslin's result.

\begin{introthm}
	\label{ThmB}
Assume that  \eqref{diag:ring-spectra} is a pullback square of $\E_{1}$-ring spectra all of which are connective. If the multiplication map $A' \otimes_{A} B \to B'$ is $n$-connective for some $n \geq 1$,  then the induced diagram of $K$-theory spectra 
\[ 
\begin{tikzcd}
	K(A) \ar[r] \ar[d] & K(B) \ar[d] \\ 
	K(A') \ar[r] & K(B')
\end{tikzcd}
\]
is $n$-cartesian (see Definition~\ref{connectivity}).
There is also a more general statement for  $K$-theory with coefficients and for $k$-connective localizing invariants, see Theorems~\ref{thm:Suslin-general}, \ref{thm:Suslin}.
\end{introthm}

If \eqref{diag:ring-spectra} is a Milnor square, then $A' \otimes_{A}  B \to B'$ is $n$-connective if and only if $\Tor_{i}^{A}(A', B) = 0$ for $i=1, \dots, n-1$. Thus \cref{ThmB} strengthens Suslin's result as the condition only depends on the given Milnor square and not on the ideal $I = \ker(A \to A')$ as an abstract non-unital ring. For instance, for $i=1$, Suslin's condition is that $I=I^{2}$, whereas $\Tor^{A}_{1}(A/I, B)$ vanishes for example if $B = A/J$ and $I\cap J = 0$, a strictly weaker condition. 

Our result also applies more generally than Suslin's, for instance to `analytic isomorphism squares' and to affine Nisnevich squares, respectively, see Example~\ref{ex:LT-squares} for details. In both cases, the multiplication map $A' \otimes_{A} B \to B'$ is an equivalence, hence $A' \wtimes{A}{B'} B \simeq B'$, and the long exact sequence \eqref{eq:long-boxtimes-sequence} is the classical one obtained from the work of Karoubi \cite{Karoubi-ai}, Quillen \cite{Quillen-Higher-K}, Vorst \cite{Vorst}, and Thomason \cite{ThomasonTrobaugh}, respectively. 
 
 \medskip
The \textit{second approach} is to use trace methods. Here, the decisive step was done by Corti{\~n}as.
Combining a pro-version of the result of Suslin--Wodzicki \cite{Suslin-Wod} with ideas of Cuntz and Quillen from their proof of excision for periodic cyclic homology over fields of characteristic zero \cite[Theorem 5.3]{CQ}, he proved that  the failure of excision in $K$-theory is rationally the same as that in (negative) cyclic homology \cite[Main Theorem]{Cortinas}.
Geisser and Hesselholt \cite[Theorem 1]{GH} proved the analog of Corti{\~n}as' result with finite coefficients, replacing negative cyclic homology by topological cyclic homology. Using these results, Dundas and Kittang finally proved that the failure of excision in $K$-theory  is also integrally the same as that in topological cyclic homology under the additional assumption that both maps $A' \to B'$ and $B \to B'$ in the Milnor square \eqref{diag:ring-spectra} are surjective \cite[Theorem 1.1]{DK2}.

All these results are special cases of the following theorem, which itself is a simple  direct consequence of our Main Theorem. 
We call a localizing invariant $E$ \emph{truncating}, if the canonical map $E(A) \to E(\pi_{0}(A))$ is an equivalence for every connective $\E_{1}$-ring spectrum $A$.
Examples of truncating invariants are periodic cyclic homology $\HP(-/k)$ over commutative $\Q$-algebras~$k$, the fibre $K^{\inf}_{\Q}$ of the Goodwillie--Jones Chern character $K(-)_{\Q} \to \HN(-\otimes \Q / \Q)$ from rational $K$-theory to negative cyclic homology, the fibre $K^{\inv}$ of the cyclotomic trace $K \to \TC$ from $K$-theory to topological cyclic homology, and Weibel's homotopy $K$-theory of $H\Z$-algebras.

\begin{introthm}
	\label{truncating-invariants-are-excisive}
Any truncating invariant satisfies excision and nilinvariance. More precisely, assume that \eqref{diag:ring-spectra} is a pullback square of $\E_{1}$-ring spectra all of which are connective and that the induced map $\pi_{0}(A' \otimes_{A} B) \to \pi_{0}(B')$ is an isomorphism. If $E$ is a truncating  invariant, then the induced square
\[
\begin{tikzcd}
	E(A) \ar[r] \ar[d] & E(B) \ar[d] \\ 
	E(A') \ar[r] & E(B')
\end{tikzcd}
\]
is a pullback. Moreover, if $I$ is a two-sided nilpotent ideal in the discrete ring $A$, then $E(A) \to E(A/I)$ is an equivalence.
\end{introthm}

We refer to Corollaries~\ref{cor:Cuntz-Quillen}, \ref{cor:Cortinas}, and \ref{excision for Kinv} for the results of Cuntz--Quillen, Corti{\~n}as, and Dundas--Kittang.
In particular, this removes the additional surjectivity assumption in the result of Dundas and Kittang.

\medskip

The \textit{third approach} is based on the pro-version of Suslin's excision theorem due to Corti{\~n}as \cite[Theorem 3.16]{Cortinas} and  Geisser--Hesselholt \cite[Theorem 3.1]{GH2} and the observation of Morrow that the sufficient pro-Tor vanishing condition for excision  is automatically satisfied for ideals in noetherian commutative rings \cite[Theorem 0.3]{Morrow-pro-unitality}. A similar pro-excision result for noetherian simplicial commutative rings was established by Kerz--Strunk--Tamme in order to prove pro-descent for abstract blow-up squares in $K$-theory \cite[Theorem 4.11]{KST}.
Again, our Main Theorem easily implies a  common generalization of all of these pro-excision results: Theorem~\ref{ThmB} holds mutatis mutandis for pro-systems of $\E_{1}$-ring spectra, see Theorem~\ref{thm:pro-excision} for details.
In particular, if $A \to B$ is a map of discrete rings sending the ideal $I \subseteq A$ isomorphically onto the ideal $J \subseteq B$, then the  diagram of pro-rings
\[
\begin{tikzcd}
 A \ar[d]\ar[r] & B \ar[d] \\ 
 \{A/I^{\lambda}\} \ar[r] & \{ B/J^{\lambda}\} 
\end{tikzcd}
\]
where $\lambda\in \N$ induces a weakly cartesian square of $K$-theory pro-spectra provided the pro-Tor-groups $\{\Tor^{A}_{i}(A/I^{\lambda}, B)\}$ vanish for $i>0$.
 Again, this condition is strictly weaker than the one occurring in \cite[Theorem 3.1]{GH2} and is automatic if $A$ is commutative and noetherian. See Corollaries~\ref{pro excision GH} and \ref{first cor pro exc} for details.

\medskip

A key feature of our Main Theorem applied to $K$-theory is that it describes the failure of excision as the relative $K$-theory of the $\E_{1}$-map $A' \wtimes{A}{B'} B \to B'$. Note that for a Milnor square, this map is just a 0-truncation.
 Besides the applications already mentioned, we can use this observation to obtain qualitative statements about the torsion in birelative algebraic $K$-groups improving a previous result of Geisser and Hesselholt \cite[Theorem C]{GH2}.
We are grateful to Akhil Mathew for pointing out that our method should give a direct proof of this result avoiding topological cyclic homology.

\begin{introthm}
Assume that \eqref{diag:ring-spectra} is a pullback square of $\E_1$-ring spectra all of which are connective. If the map $A' \otimes_A B \to B'$ is 1-connective, and if the homotopy groups of its fibre in degrees $\leq n$ are annihilated by  $N$ for some integer $N >0$, then the birelative $K$-groups $K_{i}(A,B,A',B')$ are annihilated by some power of $N$ for every $i \leq n$.
\end{introthm}
If \eqref{diag:ring-spectra} is a Milnor square, the condition in the theorem is that the groups $\Tor_{i}^{A}(A',B)$ are killed by multiplication by $N$. 
Using our Main Theorem again, we can reduce questions about relative $K$-groups of nilpotent ideals to questions about relative $K$-groups of 0-truncations. We obtain the following  generalization of a result of Geisser--Hesselholt \cite[Theorem A]{GH2}.

\begin{introthm}
	\label{intro-thm:D}
Let $A$ be a discrete ring, and let $I$ be a two-sided nilpotent ideal in $A$. Assume that $N\cdot I = 0$ for some integer $N > 0$. Then for every integer $i$, the relative K-groups $K_{i}(A,A/I)$ are annihilated by some power of $N$.
\end{introthm}

In both theorems the assumptions are automatically satisfied if \eqref{diag:ring-spectra} is a Milnor square of $\Z/N$-algebras. In this situation the previous two theorems are due to Geisser and Hesselholt. Their method does not apply directly to our more general setting.

\medskip

Our methods combined with an argument of \cite{KST} give the following cdh-descent result, which we prove in Appendix~\ref{sec:cdh}.
Various special cases of it have been known previously, see  Appendix~\ref{sec:cdh} for a discussion.
\begin{introthm}
Any truncating invariant satisfies cdh-descent on quasi-compact, quasi-separated schemes. 
\end{introthm}
Finally, in Appendix~\ref{app:B} we study localizing invariants with coefficients in some auxiliary small, stable, idempotent complete $\infty$-category and prove cdh-descent for truncating invariants with coefficients in some quasi-coherent sheaf of connective algebras over the base scheme.

\medskip

To end this introduction, we sketch the proof of the Main Theorem.
The condition that \eqref{diag:ring-spectra} is a pullback square implies that $\Perf(A)$ embeds as a full subcategory in the lax pullback
\[
\Perf(A') \laxtimes{\Perf(B')} \Perf(B).
\]
It turns out that the Verdier quotient of this lax pullback by $\Perf(A)$ is generated by a single object, and that the underlying spectrum of the endomorphisms of this generator is equivalent to $A' \otimes_{A} B$. This easily implies the result.

As side results we also obtain criteria for the diagram of (perfect) modules associated with \eqref{diag:ring-spectra} to be cartesian, see Proposition~\ref{prop:pullback-module-cats}, and a proof of the fundamental theorem for localizing invariants via the Toeplitz ring in analogy to Cuntz's proof of Bott periodicity, see Corollary~\ref{cor:fundamental-theorem}.

\medskip

\noindent \textsc{Acknowledgements.} We would like to thank Moritz Kerz, Hoang Kim Nguyen, and Christoph Schrade for useful discussions. We are grateful to Akhil Mathew for his comments on our work. In particular, he pointed out the possibility to use our results to give a direct proof of Geisser and Hesselholt's results on torsion in relative and birelative $K$-groups. We further thank Lars Hesselholt, Thomas Nikolaus, and Oscar Randal-Williams for helpful comments on a draft of this paper. Finally, we thank the anonymous referees. They suggested some simplifications in the section about pro-excision, to formulate Lemma~\ref{rings vs 1-generated cats}, which makes the proof of the main theorem  clearer, as well as to include homotopy $K$-theory as an example of  a truncating invariant,  and made several further suggestions which helped to improve the paper.

\section{The main theorem}

\begin{Notation}
We write   $\Cat_\infty^\perf$ for the $\infty$-category of small, stable, idempotent complete $\infty$-categories and exact functors.
Given an $\E_1$-ring spectrum $A$, we denote by $\RMod(A)$ the presentable, stable $\infty$-category of $A$-right modules. The $\infty$-category of \emph{perfect} $A$-modules, $\Perf(A) \in \Cat_\infty^\perf$, is by definition the smallest stable full subcategory of $\RMod(A)$ which contains $A$ and is closed under retracts. It coincides with the compact objects in $\RMod(A)$ and its ind-completion $\Ind(\Perf(A))$ is equivalent to $\RMod(A)$.
\end{Notation}

\begin{dfn} \label{localizing invariants}
A sequence $\cA \to \cB \to \cC$ in $\Cat^{\perf}_{\infty}$ is called \emph{exact} if the composite is zero, $\cA \to \cB$ is fully faithful, and the induced functor on the idempotent completion of the Verdier quotient $\Idem(\cB/\cA) \to \cC$ is an equivalence. Equivalently, the sequence is a fibre and a cofibre sequence in $\Cat^{\perf}_{\infty}$. We refer to \cite[\S 5]{BGT} for a general discussion of exact sequences in $\Cat^{\perf}_{\infty}$.
For $\cT$ a stable $\infty$-category, a $\cT$-valued \emph{localizing invariant}\footnote{Such an invariant was called \emph{weakly localizing} in \cite{Tamme:excision}, as in \cite{BGT} a localizing invariant is further required to commute with filtered colimits. However, we think that the decisive property is that of sending exact sequences to fibre sequences and propose to use another adjective, e.g.\ finitary, for the localizing invariants that preserve filtered colimits.}  is a functor 
\[
E\colon \Cat^{\perf}_{\infty} \lto \cT,
\]
which sends exact sequences in $\Cat^{\perf}_{\infty}$ to fibre sequences in $\cT$.
\end{dfn}

Let us  recall the setup for our main theorem.
We consider a commutative diagram of $\E_1$-ring spectra as follows. 
\begin{equation}
	\label{diag:ring-spectra2}\tag{$\square$}
		\begin{tikzcd}
		A \ar[r] \ar[d] & B \ar[d] \\ 
		A' \ar[r] & B'
		\end{tikzcd}
\end{equation}

For convenience we recall the statement of the Main Theorem.

\begin{thm}[Main Theorem]\label{main theorem}  
Assume that \eqref{diag:ring-spectra2} is a pullback square of $\E_1$-ring spectra. Associated to this square there exists a natural $\E_1$-ring spectrum $A' \wtimes{A}{B'} B$ with the following properties: The original diagram  \eqref{diag:ring-spectra} extends to a  commutative diagram of $\E_{1}$-ring spectra
\begin{equation}
	\label{diag:natural-diagram}
	\begin{tikzcd}
	 A \ar[d]\ar[r] & B \ar[d] \ar[ddr, bend left=20] & \\ 
	 A' \ar[r] \ar[drr, bend right=15] & A' \wtimes{A}{B'} B \ar[dr] & \\
	 & & B'
	\end{tikzcd}
\end{equation}
and any localizing invariant sends the inner square in~\eqref{diag:natural-diagram} to a pullback square. 

The underlying spectrum of $A' \wtimes{A}{B'} B$ is equivalent to $A'\otimes_A B$, and the underlying diagram of spectra is the canonical one.
\end{thm}

The following is an immediate consequence of the Main Theorem. 
\begin{cor}\label{cor intro}
Assume that \eqref{diag:ring-spectra2} is a pullback square of $\E_1$-ring spectra. If the multiplication map $A' \otimes_{A} B \to B'$ is an equivalence, then  the square
\[
\begin{tikzcd}
 E(A) \ar[d]\ar[r] & E(B) \ar[d] \\ 
 E(A') \ar[r] & E(B') 
\end{tikzcd}
\]
is cartesian for any localizing invariant $E$.
\end{cor}

If the map $A \to A'$ is Tor-unital, i.e.\ if the multiplication $A'\otimes_A A' \to A'$ is an equivalence, then also the map $A'\otimes_A B \to B'$ is an equivalence, see \cref{lemma Tor-unitality}. Thus \cref{cor intro} recovers a previous result of the second author, see \cite[Theorem 28]{Tamme:excision}.

\medskip

For the proof of \cref{main theorem} we need some preparations.
We begin by  recalling the notion of lax pullbacks of $\infty$-categories \cite[Definition~5]{Tamme:excision} and refer to \cite[Section 1]{Tamme:excision} for more details and references.
Let us consider a  diagram 
\begin{equation}\label{diag:cats}
\begin{tikzcd}
	& \cB \ar[d,"q"] \\ \cA \ar[r,"p"] & \cC 
\end{tikzcd}
\end{equation}
of  $\infty$-categories.
The \emph{lax pullback} of \eqref{diag:cats}, denoted by $\cA \laxtimes{\cC} \cB$, is defined via the pullback diagram
\begin{equation}\label{diag:cats2}
\begin{tikzcd}
	\cA \laxtimes{\cC} \cB \ar[r] \ar[d] & \Fun(\Delta^1,\cC) \ar[d] \\ 
	\cA \times \cB \ar[r,"p\times q"] & \cC \times \cC
\end{tikzcd}
\end{equation}
in simplicial sets. 
Since the source-target map $\Fun(\Delta^1,\cC) \to \cC\times \cC$ is a categorical fibration, $\cA \laxtimes{\cC} \cB$ is in fact an $\infty$-category, and \eqref{diag:cats2} is a pullback diagram of $\infty$-categories. Thus, the objects of $\cA\laxtimes{\cC} \cB$ are triples $(X,Y,f)$ with $X \in \cA$, $Y \in \cB$, and $f \colon p(X) \to q(Y)$ a morphism in $\cC$.

If \eqref{diag:cats} is a diagram of stable $\infty$-categories and exact functors, then  $\cA \laxtimes{\cC} \cB$ is stable, and the projection functors to $\cA$, $\cB$, and $\Fun(\Delta^{1}, \cC)$, respectively, are exact \cite[Lemma 8]{Tamme:excision}. 
Moreover, colimits  in $\cA \laxtimes{\cC} \cB$ are formed component-wise. There are fully faithful inclusions $j_{1}\colon \cA \to \cA \laxtimes{\cC} \cB$ and $j_{2}\colon \cB \to \cA \laxtimes{\cC} \cB$ given by $X \mapsto (X,0,0)$ and $Y \mapsto (0,Y,0)$, respectively, and any object $(X,Y,f)$ of $\cA \laxtimes{\cC} \cB$ sits in a fibre sequence 
\begin{equation}
	\label{seq:canonical-fibre-sequence}
(0,Y,0) \lto (X,Y,f) \lto (X,0,0).
\end{equation}

The following lemma will be used in the proof of Proposition~\ref{prop:identification-boxtimes-A} below.
\begin{lemma}
	\label{lemma:right-adjoint-of-A-in-lax-pb}
Assume that \eqref{diag:cats} is a diagram  of stable $\infty$-categories and exact functors and suppose that $p$ admits a right adjoint $u$. Then the functor $j_{1}\colon \cA \to \cA \laxtimes{\cC} \cB$ has a right adjoint given by $(X,Y,f) \mapsto \fib(X \to uq(Y))$ where the map in the fibre is the composition of the unit $X \to up(X)$ with $u(f)\colon up(X) \to uq(Y)$.
\end{lemma}
\begin{proof}
A unit transformation is given by the canonical equivalence $X \simeq \fib(X \to 0)$.
Indeed, this follows easily from the formula for mapping spaces in the lax pullback, see for instance \cite[Remark 6]{Tamme:excision}.
\end{proof}

We now consider a commutative square \eqref{diag:ring-spectra2} of $\E_1$-ring spectra.

\begin{lemma}
	\label{lemma:adjoint of inclusion in lax pullback}
The  functor
\[
i\colon \RMod(A) \lto \RMod(A') \laxtimes{\RMod{(B')}} \RMod(B)
\]
induced by extension of scalars admits a right adjoint $s$.  
Explicitly, for an object $(M,N,f) \in \RMod(A') \laxtimes{\RMod{(B')}} \RMod(B)$ we have
\[
s(M,N,f) \simeq M \times_{N \otimes_{B} B'} N
\]
where the map $M \to N \otimes_{B} B'$ is the composite $M \to M \otimes_{A'} B' \xrightarrow{f} N \otimes_{B} B'$.
\end{lemma}

\begin{proof}
Both categories are presentable, and the functor $i$ preserves colimits \cite[Lemma 8]{Tamme:excision}, thus $i$ admits a right adjoint $s$.
To obtain the explicit formula for $s$, we observe that we have canonical equivalences of mapping spectra
\[
s(M,N,f) \simeq \map_{A}(A, s(M,N,f) ) \simeq \map_{\slax} ( i(A), (M,N,f) )
\]
where the second equivalence comes from the adjunction. We may identify $i(A) = (A', B, \id)$. By the formula for mapping spaces in a lax pullback we have a pullback square
\[
\begin{tikzcd}
 \map_{\slax}(i(A), (M,N,f) ) \ar[d]\ar[r] & \map_{\RMod(B')^{\Delta^{1}}}( \id, f)  \ar[d] \\ 
 \map_{A'}(A' ,M) \times \map_{B}(B,N) \ar[r] & \map_{B'}(B', M \otimes_{A'} B' ) \times \map_{B'}(B', N \otimes_{B} B').
\end{tikzcd}
\]
In this square, the lower left corner identifies with $M \times N$, the lower right corner with $M \otimes_{A'} B' \times N \otimes_{B} B'$.
The mapping spectrum in the upper right corner canonically identifies with $M \otimes_{A'} B'$. Under these identifications, the right vertical map is given by $(\id, f)$. We thus get the left pullback square in the following diagram.
\[
\begin{tikzcd}
 s(M,N,f) \ar[d]\ar[r] & M \otimes_{A'} B' \ar[d, "{(\id,f)}"] \ar[r]			&  0 \ar[d]\\ 
 M \times N \ar[r] & M\otimes_{A'} B'  \times  N \otimes_{B} B' \ar[r, "{f-\id}"] & N\otimes_{B} B' 
\end{tikzcd}
\]
Obviously, the right-hand square is also a pullback square, whence is the combined square.  This proves the claim.
\end{proof}

\begin{lemma}
	\label{lemma:fully faithful}
Assume that \eqref{diag:ring-spectra2} is a pullback square of $\E_1$-ring spectra.
Then the functor $i$ in \cref{lemma:adjoint of inclusion in lax pullback} and hence its restriction to perfect modules
\[ 
	i\colon \Perf(A) \lto \Perf(A') \laxtimes{\Perf{(B')}} \Perf(B) 
\]
are fully faithful.
\end{lemma}
\begin{proof}
It suffices to show that the unit map $X \to s   i(X)$ is an equivalence for every $A$-module $X$. Using the concrete formulas, this follows directly from the fact that \eqref{diag:ring-spectra2} is a pullback diagram. Alternatively, one can check directly that the induced map on mapping spectra is an equivalence by 
 reducing  to the case $X=A$, where it is again immediate from the fact that \eqref{diag:ring-spectra2} is a pullback. 
\end{proof}

\begin{dfn}
A small, stable, idempotent complete $\infty$-category $\cA$ is said to be \emph{generated by a set of objects} $S$, if $\cA$ coincides with the smallest stable, idempotent complete  full subcategory of $\cA$ that contains $S$.
\end{dfn}

\begin{lemma}\label{lemma:generators lax pullback}
Consider a diagram as \eqref{diag:cats} in $\Cat^{\perf}_{\infty}$. If $\cA$ is generated by the set of objects $S$ and $\cB$ is generated by the set of objects $T$, then the lax pullback $\cA \laxtimes{\cC} \cB$ is generated by the set of objects
\[ \big\{ (X,0,0) \;|\; X \in S \big\} \cup \{ (0,Y,0) \;|\; Y \in T \big\}.\]
\end{lemma}
\begin{proof}
It suffices to recall that every object $(M,N,f)$ in $\cA \laxtimes{\cC} \cB$ sits in a cofibre sequence
\[(0,N,0) \to (M,N,f) \to (M,0,0). \qedhere \]
\end{proof}

Write $\Pr^{\mathrm{L}}$ for the symmetric monoidal $\infty$-category of presentable $\infty$-categories and left adjoint functors \cite[4.8.1.15]{LurieHA}, $\Pr^{\mathrm{L, st}}$ for its full subcategory of stable, presentable $\infty$-categories, and $\Pr^{\mathrm{L, st}}_\omega$ for its subcategory of compactly generated stable $\infty$-categories and functors that preserve compact objects. The tensor product $\cA \otimes \cB$ in $\Pr^{\mathrm{L}}$ is the initial object equipped with a functor from $\cA \times \cB$ which preserves colimits in each variable.
As a smashing localization of $\Pr^{\mathrm{L}}$, the $\infty$-category $\Pr^{\mathrm{L, st}}$ inherits a symmetric monoidal structure \cite[Section 4.8.2]{LurieHA} with tensor unit the $\infty$-category $\Sp$ of spectra. As $\Sp$ is compactly generated, and the tensor product of two compactly generated, stable $\infty$-categories is again so \cite[Lemma 5.3.2.11]{LurieHA}, $\Pr^{\mathrm{L, st}}_\omega \subseteq \Pr^{\mathrm{L, st}}$ is a symmetric monoidal subcategory. The equivalence $\Cat_\infty^\mathrm{perf} \simeq \Pr^{\mathrm{L, st}}_\omega$, given by taking ind-completions and compact objects, induces a symmetric monoidal structure on $\Cat_\infty^\mathrm{perf}$. The tensor product $\cA \otimes \cB$ in this case is the initial object of $\Cat_\infty^\mathrm{perf}$ equipped with a functor from $\cA \times \cB$ which is exact in each variable, the tensor unit is the $\infty$-category $\Sp^{\omega}$ of finite spectra.

Recall furthermore that an $\E_{0}$-algebra in a symmetric monoidal $\infty$-category is just an object together with a map from the tensor unit \cite[Remark 2.1.3.10]{LurieHA}. Thus the $\infty$-category $\mathrm{Alg} _{\E_{0}}(\Cat^{\perf}_{\infty})$ may be identified with the $\infty$-category of pairs $(\cC, C)$ where $\cC \in \Cat^{\perf}_{\infty}$ and $C$ is an object of $\cC$. Morphisms in $\mathrm{Alg} _{\E_{0}}(\Cat^{\perf}_{\infty})$ are functors $F\colon \cC \to \cC'$ equipped with an equivalence $F(C) \simeq C'$.

Our proof of the main theorem relies  on the Schwede--Shipley recognition theorem in its following form.
\begin{lemma}\label{rings vs 1-generated cats}
The association $R \mapsto (\Perf(R),R)$ extends to a fully faithful functor
\[ 
\mathrm{Alg}_{\E_1}(\Sp) \lto \mathrm{Alg}_{\E_0}(\Cat^{\perf}_\infty).
\]
Its essential image consists of the pairs $(\cC,C)$ for which $C$ generates $\cC$, and its inverse is given by the functor sending $(\cC,C)$ to the  endomorphism ring spectrum $\End_{\cC}(C)$.
\end{lemma}
\begin{proof}
The composite
\[ 
\mathrm{Alg}_{\E_1}(\Sp) \lto  \mathrm{Alg}_{\E_0}(\Cat^\mathrm{perf}_\infty) \overset{\simeq}{\lto} \mathrm{Alg}_{\E_0}(\Pr^{\mathrm{L, st}}_\omega) \lto \mathrm{Alg}_{\E_0}(\Pr^{\mathrm{L}}) 
\]
is the functor $R \mapsto (\RMod(R), R)$. By \cite[Proposition 7.1.2.6]{LurieHA} it is fully faithful with essential image those pairs $(\cC,C)$ for which $C$ is  compact  and generates $\cC$ under colimits, or equivalently $C$ lies in $\cC^{\omega}$ and is a generator in our sense.
Since the last functor is fully faithful on the subcategory of pairs $(\cC, C)$ where $C$ is a compact generator of $\cC$, this implies the first two assertions. 
The description of the inverse  follows from  \cite[Remark 7.1.2.3]{LurieHA}.
\end{proof}

Now we start with the proof of Theorem~\ref{main theorem}.
So suppose that the square of $\E_{1}$-rings \eqref{diag:ring-spectra2} is a pullback. 
We write $i_{1}\colon \Perf(A) \to \Perf(A')$ and $i_{2}\colon \Perf(A) \to \Perf(B)$ for the respective extension of scalars functors, and 
\begin{align*}
&j_{1}\colon \Perf(A') \lto \Perf(A') \laxtimes{\Perf(B')} \Perf(B), \\
&j_{2}\colon \Perf(B) \lto \Perf(A') \laxtimes{\Perf(B')} \Perf(B)
\end{align*}
for the inclusion functors.
We  observe that the diagram 
\[
\begin{tikzcd}
	\Perf(A) \ar[r,"i_2"] \ar[d,"i_1"'] & \Perf(B) \ar[d,"j_2"] \\
	\Perf(A') \ar[r,"\Omega j_1"]  \arrow[ur, Rightarrow, shorten >= 2em, shorten <= 2em] & \Perf(A')\laxtimes{\Perf(B')} \Perf(B)
\end{tikzcd}
\]
comes equipped with a canonical natural transformation $\tau \colon \Omega j_1  i_1 \Rightarrow j_2 i_2$ induced by the canonical (rotated) fibre sequence
\begin{equation}
	\label{seq:definition-tau}
(\Omega X\otimes_A A',0,0) \overset{\tau_{X}}{\lto} (0,X\otimes_A B,0) \lto (X\otimes_A A',X\otimes_A B, \can) 
\end{equation}
of \eqref{seq:canonical-fibre-sequence}. Here $\can$ is the natural equivalence $(X\otimes_{A} A') \otimes_{A'} B'
 \simeq (X \otimes_{A} B) \otimes_{B} B'$ that makes the diagram $\Perf(\text{\ref{diag:ring-spectra2}})$ obtained by applying $\Perf$ to \eqref{diag:ring-spectra2} commute.

We  let $\cQ$ be the cofibre of the fully faithful embedding $i$ of perfect modules from \cref{lemma:fully faithful} in $\Cat^{\perf}_{\infty}$ and write $p$ for the canonical functor $\Perf(A) \laxtimes{\Perf(B')} \Perf(B) \to \cQ$.
\begin{lemma}
	\label{Q generated by}
The induced natural transformation $p(\tau) \colon p\operatorname{\Omega} j_{1} i_{1} \Rightarrow pj_{2}i_{2}$ is an equivalence.
Moreover, $\cQ$ is generated by the single object $\overline{B} = p(0,B,0)$.
\end{lemma}
\begin{proof}
Since the cofibre of $\tau_{X}$ is in the image of the  functor $i$, the induced natural transformation $p(\tau)$ is an equivalence.
By definition, $\Perf(A')$ and $\Perf(B)$ are generated by $A'$ and $B$, respectively. Thus, by \cref{lemma:generators lax pullback}, the lax pullback $\Perf(A')\laxtimes{\Perf{(B')}} \Perf(B)$ is generated by $(0,B,0)$ and $(A',0,0)$. It follows that their images $\overline{B}$ and  $p(A',0,0)$ generate $\cQ$.
But $p(\tau_{A})$ is an equivalence $p(\Omega A',0,0) \simeq \ol B$, concluding the proof.
\end{proof}

We denote the $\E_1$-ring $\End_{\cQ}(\ol{B})$ by $A'\wtimes{A}{B'} B$. 
By Lemma~\ref{rings vs 1-generated cats} we thus have a canoncial equivalence
\[
(\cQ, \ol B) \simeq ( \Perf( A' \wtimes{A}{B'} B ), A' \wtimes{A}{B'} B )
\]
which we use to identify $(\cQ, \ol B)$ in the following. Moreover,  we get a commutative diagram
\begin{equation}\label{diag:perf-natural-diagram}
\begin{tikzcd}
	\Perf(A) \ar[r,"i_2"] \ar[d,"i_1"] & \Perf(B) \ar[d,"p j_2"] \\
	\Perf(A') \ar[r,"p\Omega j_1"] & \Perf(A'\wtimes{A}{B'} B) 
\end{tikzcd}
\end{equation}
witnessed by the natural equivalence $p(\tau)$ from  Lemma~\ref{Q generated by}, and an exact sequence
\begin{equation}
	\label{eq:fundamental-exact-sequence} 
\begin{tikzcd}
	\Perf(A) \ar[r, "i"] & \Perf(A') \laxtimes{\Perf{(B')}} \Perf(B) \ar[r,"p"] & \Perf(A'\wtimes{A}{B'} B) 
\end{tikzcd}
\end{equation}
in $\Cat_\infty^\perf$. Applying Lemma~\ref{rings vs 1-generated cats} to diagram \eqref{diag:perf-natural-diagram} gives the inner square of diagram~\eqref{diag:natural-diagram}.

Given any localizing invariant $E$, we want to show that $E$ applied to the square \eqref{diag:perf-natural-diagram}  gives a pullback square. Since $E(\Omega)$ is multiplication by $-1$, an equivalent statement is that the sequence
\begin{equation}	\label{eeee}
E(\Perf(A)) \xrightarrow{(i_{1},i_{2})} E(\Perf(A')) \oplus E(\Perf(B)) \xrightarrow{p j_{1}+p j_{2}} E(\Perf( A' \wtimes{A}{B'} B ))
\end{equation}
is a fibre sequence. But applying $E$ to the exact sequence \eqref{eq:fundamental-exact-sequence} yields a fibre sequence
\begin{equation}	\label{ffff}
E(\Perf(A)) \lto E( \Perf(A') \laxtimes{\Perf(B')} \Perf(B) )  \lto E( \Perf( A' \wtimes{A}{B'} B ) )
\end{equation}
and using \cite[Proposition 10]{Tamme:excision} we  get an equivalence
\[ 
E(\Perf(A')) \oplus E(\Perf(B)) \xrightarrow{\simeq} E( \Perf(A') \laxtimes{\Perf(B')} \Perf(B) )
\]
induced by $j_{1}$ and $j_{2}$ with inverse induced by the two projection functors. Thus the sequence \eqref{eeee} is in fact equivalent to the sequence \eqref{ffff}, and thus a fibre sequence.

We now construct the remaining parts of diagram~\eqref{diag:natural-diagram} from Theorem~\ref{main theorem} on the level of categories of perfect modules,  more precisely, the image of diagram~\eqref{diag:natural-diagram} under the fully faithful functor from Lemma~\ref{rings vs 1-generated cats}.
There is a functor
\[
c\colon \Perf(A') \laxtimes{\Perf(B')} \Perf(B) \lto \Perf(B')
\]
sending an object $(M,N,f)$ to $\cof(f)$. 
We observe that in the diagram
\[
\begin{tikzcd}
	\Perf(A) \ar[r,"i_2"] \ar[d,"i_1"] 		& \Perf(B) \ar[d,"j_2"] \ar[ddr, bend left=20,"k_2"] & \\
	\Perf(A') \ar[r,"\Omega j_1"] \ar[drr, bend right=10,"k_1"']   \arrow[ur, Rightarrow, shorten >= 2em, shorten <= 2em]     	& \Perf(A')\laxtimes{\Perf{B'}} \Perf(B) \ar[dr,"c"] \ar[d, Rightarrow, shorten <= 1ex, shorten >= 1ex] 
			\ar[r, Rightarrow, shorten <= .5ex, shorten >= 1ex] & \phantom{\Perf(B')}\\
	& \phantom{c} & \Perf(B'),
\end{tikzcd}
\]
where $k_{1}$ and $k_{2}$ are the respective extension of scalars functors,
the inner square is endowed with the transformation $\tau$, and the triangles are equipped with the natural equivalences $\alpha_{1} \colon c \operatorname{\Omega} j_1 \Rightarrow k_1$ and $\alpha_{2} \colon c j_2 \Rightarrow k_2$ given by the canonical maps
\[
\alpha_{1,M} \colon c((\Omega M,0,0)) = \cof(\Omega M\otimes_{A'} B' \to 0) \overset{\simeq}{\lto} M\otimes_{A'} B' 
\]
and
\[ 
\alpha_{2,N} \colon c((0,N,0)) = \cof(0 \to N\otimes_B B')  \overset{\simeq}{\lto} N\otimes_B B'.
\]
We  claim that pasting these natural transformations together yields the original commutative square $\Perf(\text{\ref{diag:ring-spectra2}})$ of categories of perfect modules. For this we have to show that the  diagram of natural transformations 
\[
\begin{tikzcd}[column sep=large]
	c\operatorname{\Omega} j_1  i_1(X) \ar[r,"\alpha_{1,i_{1}(X)}"] \ar[d,"c(\tau_{X})"'] & k_1 i_1(X) \ar[d,"\can"] \\
	c j_2 i_2(X) \ar[r,"\alpha_{2,i_{2}(X)}"] & k_2 i_2(X)
\end{tikzcd}
\]
commutes naturally for $X$ in $\Perf(A)$.
But by definition $\tau_{X}$  sits in a cofibre sequence
\[
(\Omega X\otimes_A A',0,0) \overset{\tau_{X}}{\lto} (0,X\otimes_A B,0) \lto (X\otimes_A A',X\otimes_A B, \can) 
\]
so that its effect on cofibres is the canonical equivalence.

Since the composition of $c$ with the inclusion of $\Perf(A)$ in the lax pullback vanishes, it factors essentially uniquely through a functor $\bar c \colon \Perf(A' \wtimes{A}{B'} B ) \to \Perf(B')$
sending the generator $A' \wtimes{A}{B'} B$ to $c((0,B,0)) \simeq B'$.
As a result, we  obtain that in the diagram
\[\begin{tikzcd}
	\Perf(A) \ar[r,"i_2"] \ar[d,"i_1"] & \Perf(B) \ar[d,"p j_2"] \ar[ddr, bend left=15,"k_2"] & \\
	\Perf(A') \ar[r,"p\Omega j_1"] \ar[drr, bend right=10,"k_1"] & \Perf(A'\wtimes{A}{B'} B) \ar[dr,"\bar c"] & \\
	& & \Perf(B')
\end{tikzcd}\]
every subdiagram commutes, witnessed by the transformations $p(\tau)$, and the transformations induced by $\alpha_{1}$ and $\alpha_{2}$, and all functors respect the preferred generators.   Furthermore, by the discussion above, the outer diagram is  the original commutative diagram $\Perf(\text{\ref{diag:ring-spectra2}})$ of categories of perfect modules. Finally, by Lemma~\ref{rings vs 1-generated cats}, we obtain the commutative diagram~\eqref{diag:natural-diagram} by passing to endomorphisms of the preferred generators. This finishes the proof of the first part of Theorem~\ref{main theorem}.

In the following we  establish the identification of $A' \wtimes{A}{B'} B$ and all involved maps, see Propositions~\ref{prop:identification-boxtimes-B}, \ref{prop:identification-boxtimes-A}, and \ref{map out of weird ring} below. In fact, in addition to what is stated in Theorem~\ref{main theorem}, we also identify the underlying $A'$- respectively $B$-bimodule structure of $A' \wtimes{A}{B'} B$.
To compute $A' \wtimes{A}{B'} B = \End_{\cQ}(\ol{B})$ we may pass to ind-completions.
Note that by \cite[Proposition 13]{Tamme:excision}, $\Ind( \Perf(A') \laxtimes{\Perf(B')} \Perf(B) ) \simeq \RMod(A') \laxtimes{\RMod{(B')}} \RMod(B)$.

\begin{lemma}
	\label{lemma:right-adjoint-of-pi}
The functor 
\[ 
	p\colon \RMod(A') \laxtimes{\RMod{(B')}} \RMod(B) \lto \Ind(\cQ)
\]
is a Bousfield localization, i.e.~it admits a fully faithful right adjoint $r$. The localization functor $L = rp$ is given by the cofibre
\[
\cof(is(-) \to (-) )
\]
of the counit transformation of the adjunction $(i,s)$ from \cref{lemma:adjoint of inclusion in lax pullback}.
\end{lemma}

\begin{proof}
Since $\cQ$ is the idempotent completion of the Verdier quotient of the lax pullback of perfect modules by $\Perf(A)$, the functor $p$ on ind-completions is a Bousfield localization and
has  kernel $\Ind(\Perf(A)) \simeq \RMod(A)$, see for instance \cite[Proposition I.3.5]{NS}. 
The local objects are given by the image of $r$.
It follows that the pair $(i(\RMod(A)), r(\Ind(\cQ)))$ is a semi-orthogonal decomposition of the $\infty$-category
\(
\RMod(A') \laxtimes{\RMod{(B')}} \RMod(B) 
\)
in the sense of \cite[Definition 7.2.0.1]{LurieSAG}, see Corollary 7.2.1.7 there. The desired formula for $L$ now follows immediately from \cite[Remark 7.2.0.2]{LurieSAG}.
\end{proof}

The following observations will be used in the proofs below. 
We consider a functor $F\colon \cC \to \cD$ between stable, presentable $\infty$-categories and we fix an object $c \in \cC$. Then we can form the endomorphism spectrum $\End_{\cC}(c)$, and $F(c)$ carries canonically the structure of an $\End_{\cC}(c)$-left module in $\cD$. If $G$ is a second functor $\cC \to \cD$, we have the two $\End_{\cC}(c)$-left modules $F(c)$, $G(c)$ in $\cD$ and the mapping spectrum $\map_{\cD}( F(c), G(c) )$ acquires a canonical $\End_{\cC}(c)$-bimodule structure, where the left module structure is induced by post-composition via $G$ and the right module structure is induced by pre-composition via $F$.

\medskip

Given a pullback square of $\E_{1}$-ring spectra  \eqref{diag:ring-spectra2}, we denote the fibre of $B \to B'$ by $I$. Thus $I$ is naturally a $B$-bimodule, and the fibre of $A \to A'$ is obtained from $I$ by the forgetful functor to $A$-bimodules.
Note that via the respective forgetful functors $I$ has a canonical $(B,A)$-bimodule structure and $B$ has a canonical $(A,B)$-bimodule structure. Hence the tensor product $I \otimes_{A} B$ carries a canonical $B$-bimodule structure.

\begin{prop}\label{1.13}
	\label{prop:identification-boxtimes-B}
The underlying spectrum of $A' \wtimes{A}{B'} B$ is canonically equivalent to $A' \otimes_{A} B$. Under this identification the  ring map $B \to A' \wtimes{A}{B'} B$ induced from the functor $p j_{2}$ is the obvious map $B \to A' \otimes_{A} B$. Moreover, the underlying $B$-bimodule of $A' \wtimes{A}{B'} B$ sits in a cofibre sequence
\[
I \otimes_{A} B \lto B \lto A' \wtimes{A}{B'} B
\]
in the $\infty$-category of $B$-bimodule spectra.
\end{prop}

\begin{proof}
It follows from  \cref{lemma:right-adjoint-of-pi} that we have an equivalence of $\E_{1}$-rings
\[
A' \wtimes{A}{B'} B = \End_{\cQ}(\ol{B}) \simeq \End_{\slax}(L(0,B,0) ).
\]
Left multiplication provides an equivalence of $\E_{1}$-rings $B \simeq \End_{B}(B)$. We still write $j_{2}$ for the functor
\[
\RMod(B) \lto \RMod(A') \laxtimes{\RMod(B')} \RMod(B)
\]
that sends $M$ to $(0,M,0)$. Then  the $\E_{1}$-map $B \to A' \wtimes{A}{B'} B$ is the map $B \simeq \End_{B}(B) \to \End_{\slax}(L(0,B,0)) = \End_{\slax}(Lj_{2}(B)) \simeq A' \wtimes{A}{B'} B$ induced by the functor $L  j_{2}$.  
Since $Lj_{2}(B)$ is a local object, the canonical map $j_{2}(B) \to Lj_{2}(B)$ induces an equivalence
\[
\End_{\slax}(Lj_{2}(B)) \xrightarrow{\simeq} \map_{\slax}( j_{2}(B), Lj_{2}(B) ).
\]
Using Lemma~\ref{lemma:right-adjoint-of-pi} we see that there is a commutative diagram 
\begin{equation}
	\label{diag:identification-ring-map}
\begin{tikzcd}
 \map_{\slax}( j_{2}(B), i sj_{2}(B) ) \ar[r] &  \map_{\slax}( j_{2}(B), j_{2}(B) ) \ar[r] &  \map_{\slax}( j_{2}(B), Lj_{2}(B) )  \\
 & \End_{B}(B) \ar[u, "{\simeq}"] \ar[r] & \End_{\slax}( Lj_{2}(B) ) \ar[u, "{\simeq}"]
\end{tikzcd}
\end{equation}
where the top row is a cofibre sequence of $B \simeq \End_{B}(B)$-bimodules, and the lower horizontal map is the $\E_{1}$-map we are interested in.
We can identify the top row further. The functor $j_{2}$ is left adjoint to the projection $\pr_{2}$ from the lax pullback to $\RMod(B)$, see for example \cite[Proposition 10]{Tamme:excision}.
From the formula for $s$ in Lemma~\ref{lemma:adjoint of inclusion in lax pullback} we see that $sj_{2}(B) \simeq I$ considered as an $A$-right module.
Moreover, for any $B$-right module $M$, we have a canonical equivalence $\map_{B}(B, M) \simeq M$ and if $M$ is a $B$-bimodule this equivalence refines to a $B$-bimodule equivalence.
Hence the first map in the  top row of \eqref{diag:identification-ring-map} canonically identifies with the $B$-bimodule map $I \otimes_{A} B \to B$ which is given by $I \to B$ and the multiplication in $B$. This establishes the last claim of the proposition.

On the other hand, tensoring the cofibre sequence of $A$-right modules $I \to A \to A'$ with $B$ we deduce  that the above map sits in a cofibre sequence
\[
I \otimes_{A} B \lto B \lto A' \otimes_{A} B
\]
where the second map is the obvious one. From this we immediately deduce the first two claims of the proposition.
\end{proof}

We now prove the analog of the previous proposition for the ring map $A' \to A' \wtimes{A}{B'} B$ induced by the functor $\operatorname{\Omega}  p j_{1}$ appearing in diagram \eqref{diag:perf-natural-diagram}. We denote the fibre of $A' \to B'$ in $A'$-bimodules by $J$. Again, after forgetting to $A$-bimodules, $J$ becomes equivalent to the fibre of $A \to B$.

\begin{prop}
	\label{prop:identification-boxtimes-A}
Under the identification of Proposition~\ref{prop:identification-boxtimes-B}, the ring map $A' \to A' \wtimes{A}{B'} B$ induced by the functor $\operatorname\Omega  p j_{1}$ is the obvious map $A' \to A' \otimes_{A} B$. Moreover, the underlying $A'$-bimodule of $A' \wtimes{A}{B'} B$ sits in a cofibre sequence 
\[
A' \otimes_{A} J \lto A' \lto A' \wtimes{A}{B'} B
\]
in the $\infty$-category of $A'$-bimodule spectra.
\end{prop}
\begin{proof}
Recall the natural transformation
\[
\tau_{A}\colon \Omega(A',0,0) \lto (0,B,0),
\]
which   becomes an equivalence upon applying the localization functor $L$ by Lemma~\ref{Q generated by}. We consider the following commutative diagram of mapping spectra in the lax pullback.
\[
\begin{small}
\begin{tikzcd}[column sep =1.5em]
\map(\Omega(A',0,0), L\Omega(A',0,0) )		& \map( (0,B,0), L\Omega(A',0,0) ) \ar[l, "{\tau_{A}^{*}}"', "\simeq"] \ar[r, "{\tau_{A,*}}", "\simeq"']	& \map( (0,B,0), L(0,B,0) )   	\\
\map(\Omega(A',0,0), \Omega(A',0,0)) \ar[u] 	& \map( (0,B,0), \Omega(A',0,0) ) \ar[l, "{\tau_{A}^{*}}"'] \ar[r, "{\tau_{A,*}}"]  \ar[u]				& \map( (0,B,0), (0,B,0) ) \ar[u]	\\
\map(\Omega(A',0,0), is\Omega(A',0,0)) \ar[u] 	& \map( (0,B,0), is\Omega(A',0,0) ) \ar[l, "{\tau_{A}^{*}}"'] \ar[r, "{\tau_{A,*}}"]  \ar[u]				& \map( (0,B,0), is(0,B,0) ) \ar[u]	\\
\end{tikzcd}
\end{small}
\]
Note that the columns in this diagram are cofibre sequences.
The inclusion of $\RMod(A')$ in the lax pullback and the functor $\Omega$ induce a canonical equivalence 
\[
A' \simeq \End_{A'}(A') \simeq \map(\Omega(A',0,0), \Omega(A',0,0) ).
\]
In the proof of Proposition \ref{prop:identification-boxtimes-B} we have identified 
\[
A' \wtimes{A}{B'} B \simeq \map((0,B,0), L(0,B,0) ) \simeq A' \otimes_{A} B.
\]
Under these identifications, the map $A' \to A' \wtimes{A}{B'} B$ induced by the functor $\operatorname\Omega  p j_{1}$ is the map from the middle term in the left column to the top right corner in the above diagram.
We now identify all terms in the above diagram. Firstly, $s\Omega(A',0,0)$ is $\Omega A'$ viewed as $A$-right module. Using Lemma~\ref{lemma:right-adjoint-of-A-in-lax-pb} and the canonical identification $\map_{A'}(\Omega A', \Omega X) \simeq X$ for any $A'$-right module (or bimodule) $X$, we identify the  lower left vertical map as the canonical map $A' \otimes_{A} J \to A'$ in $A'$-bimodules. The underlying spectrum of its cofibre is canonically equivalent to $A' \otimes_{A} B$ via the obvious map $A' \to A' \otimes_{A} B$. Hence the proof of the proposition is finished once we show that under the respective identifications the top row of the above diagram becomes the identity of $A' \otimes_{A} B$.
To see this, we observe that the two lower pullback squares of the diagram canonically identify with the pullback squares
\[
\begin{tikzcd}
A' 					& 0 \ar[l] \ar[r] 						& B \\
A' \otimes_{A} J \ar[u]	& \Omega A' \otimes_{A} B \ar[u]\ar[l]\ar[r]	& I \otimes_{A} B \ar[u]
\end{tikzcd}
\]
where the left-hand square is the cofibre sequence $\Omega B \to J \to A$ tensored with $A'$, the right-hand square is the cofibre sequence $\Omega A' \to I \to A$ tensored with $B$. This implies that the induced map on vertical cofibres is the identity of $A' \otimes_{A} B$, as desired.
\end{proof}

\begin{prop}\label{map out of weird ring}
Under the identification of Proposition \ref{prop:identification-boxtimes-B}, the $\E_{1}$-map $A' \wtimes{A}{B'} B \to B'$ is given by the map $A' \otimes_{A} B \to B'$ induced by the maps $A' \to B'$, $B \to B'$, and the multiplication in $B'$.
\end{prop}
\begin{proof}
We have already shown that diagram \eqref{diag:natural-diagram} commutes. Hence the  composites  $A' \to A' \wtimes{A}{B'} B \to B'$  and $B \to A' \wtimes{A}{B'} B \to B'$ are the given maps $A' \to B'$ and $B \to B'$, respectively.
Since as an $\E_{1}$-map $A' \wtimes{A}{B'} B \to B'$ is in particular $A'$-left linear, and since under the identification of $A' \wtimes{A}{B'} B$ with $A' \otimes_{A} B$, the $A'$-left module structure of $A' \wtimes{A}{B'} B$ is the canonical one by Proposition~\ref{prop:identification-boxtimes-A}, the map $A' \wtimes{A}{B'} B \to B'$ is the unique $A'$-linear extension of the  composite $B \to A' \wtimes{A}{B'} B \to B'$. 
Together with the above, this finishes the proof.
\end{proof}

\begin{rem}
	\label{rem:multiplication-cofibre-sequence}
For later use we record the following consequence. Assume as before that \eqref{diag:ring-spectra2} is a pullback square of $\E_{1}$-ring spectra, and write $I = \fib( B \to B' )$. The map $A' \wtimes{A}{B'} B \to B'$ is also the unique $B$-right linear extension of the given map $A' \to B'$. Hence it sits in a diagram of cofibre sequences
\[
\begin{tikzcd}
I \otimes_{A} B \ar[r] \ar[d] & B \ar[r]\ar[d, equal] & A' \wtimes{A}{B'} B \ar[d] \\
I \ar[r] & B \ar[r] & B'
\end{tikzcd}
\]
where the upper sequence is the one of Proposition~\ref{prop:identification-boxtimes-B} and the left vertical map is the $B$-right module structure on $I$.
\end{rem}

All constructions made above are  natural in the input diagram \eqref{diag:ring-spectra2}. This finishes the proof of Theorem~\ref{main theorem}.

\medskip

With the methods of this section we can also prove the following result, which is of independent interest. We are grateful to a referee for suggesting this.
Compare \cite[Theorem 16.2.0.2]{LurieSAG} for a similar result.  

\begin{prop}
	\label{prop:pullback-module-cats}
Assume that \eqref{diag:ring-spectra2} is a pullback square of $\E_1$-ring spectra. If the functor $\Perf(A'\wtimes{A}{B'} B) \to \Perf(B')$, respectively the functor $\RMod(A'\wtimes{A}{B'} B) \to \RMod(B')$, is conservative, then the square of $\infty$-categories
\begin{center}
\begin{tikzcd}
\Perf(A) \ar[d]\ar[r] & \Perf(B)\phantom{,} \ar[d]       \\ 
 \Perf(A') \ar[r] & \Perf(B'),  
\end{tikzcd}
\hspace{1em} respectively \hspace{1em}
\begin{tikzcd}
 \RMod(A) \ar[d]\ar[r] & \RMod(B)\phantom{,} \ar[d] \\ 
 \RMod(A') \ar[r] & \RMod(B'), 
\end{tikzcd}
\end{center}
is cartesian.
\end{prop}

Obviously, both functors in question are conservative if the canonical map $A' \otimes_{A} B \to B'$ is an equivalence.
For the case of perfect modules one has for instance the following criterion: If $R \to R'$ is a map of connective $\E_{1}$-rings which on $\pi_{0}$ induces a surjection with kernel contained in the Jacobson radical of $\pi_{0}(R)$, then $\Perf(R) \to \Perf(R')$ is conservative, i.e.~detects the zero object:
Given a non-trivial perfect $R$-module $P$, its lowest non-trivial homotopy group, say $\pi_{k}(P)$, is finitely generated over $\pi_{0}(R)$. Hence $\pi_k(P\otimes_R R') \cong \pi_k(P)\otimes_{\pi_0(R)} \pi_0(R')$ is also non-zero by Nakayama's lemma \cite[Proposition III.2.2]{Bass}.
In many examples the map $A' \wtimes{A}{B'} B \to B'$ induces an isomorphism on $\pi_{0}$, see for instance Example~\ref{ex:LT-squares}.

\begin{proof}[Proof of Proposition~\ref{prop:pullback-module-cats}]
Consider the following commutative diagram.
\[\begin{tikzcd}
	\Perf(A) \ar[r] \ar[d] & \Perf(A') \laxtimes{\Perf(B')} \Perf(B) \ar[r] \ar[d,equal] & \Perf(A' \wtimes{A}{B'} B) \ar[d] \\
	 \Perf(A') \times_{\Perf(B')} \Perf(B) \ar[r] &  \Perf(A') \laxtimes{\Perf(B')} \Perf(B) \ar[r,"c"] & \Perf(B')
\end{tikzcd}\]
The upper horizontal sequence is exact. The first functor  in the lower sequence identifies the pullback with the kernel of the cofibre functor $c$.
Hence the left vertical functor is fully faithful and it suffices to argue that it is essentially surjective. A diagram chase shows this to be the case if the functor $\Perf(A' \wtimes{A}{B'} B) \to \Perf(B')$ is conservative. The case of full module categories follows analogously by replacing $\Perf$ with $\RMod$ throughout.
\end{proof}

\begin{rem}
	\label{linear categories}
Let $\cE$ be an object of $\CAlg(\Cat_\infty^\perf)$, i.e.~a small, idempotent complete, stably symmetric monoidal $\infty$-category. For simplicity, we assume that $\cE$ is rigid, i.e.\ that every object of $\cE$ is dualizable.
We let $\Mod_{\cE}(\Cat_\infty^\perf)$ be the $\infty$-category of $\cE$-\emph{linear} $\infty$-categories, which we denote by $\Cat_\infty^\cE$. 
For an $\E_\infty$-ring spectrum $k$, we will write $\Cat_\infty^k$ for the $\infty$-category of $\Perf(k)$-linear categories, and refer to its objects simply as $k$-linear $\infty$-categories. Notice that $\Perf(k)$ is rigid.
A typical example of a $k$-linear $\infty$-category is the $\infty$-category $\Perf(A)$ of perfect $A$-modules for a $k$-algebra~$A$. 
The canonical forgetful functor
\[ \Cat_\infty^\cE \lto \Cat_\infty^\perf \]
preserves finite limits and finite colimits. A sequence $\cA \to \cB \to \cC$ of $\cE$-linear $\infty$-categories is called exact if it is a fibre and a cofibre sequence in $\Cat^{\cE}_{\infty}$, or equivalently if it is exact after forgetting the $\cE$-linear structure. For an auxiliary stable $\infty$-category $\cT$, a $\cT$-valued \emph{localizing invariant of $\cE$-linear $\infty$-categories} is then  a functor 
\[ 
\Cat_\infty^\cE \lto \cT 
\]
that sends exact sequences of $\cE$-linear $\infty$-categories to fibre sequences in $\cT$. 
We refer to \cite[\S\S 4\&5]{HoyoisScherotzkeSibilla} for a detailed treatment of localizing invariants in the $\cE$-linear setting. 
For a fixed $\E_\infty$-ring $k$, examples of localizing invariants of $k$-linear $\infty$-categories are provided by topological Hochschild homology relative to $k$, or Weibel's  homotopy $K$-theory for $H\Z$-linear categories, see  \cref{sec:applications to truncating invariants}.

In a $k$-linear  $\infty$-category, the mapping spectra canonically refine to $k$-module spectra, i.e.~any $k$-linear  $\infty$-category is enriched in the presentably symmetric monoidal, stable $\infty$-category $\Mod(k)$.
The analog of the Schwede--Shipley theorem (Lemma~\ref{rings vs 1-generated cats}) also holds in the $k$-linear context: The association $R \mapsto (\Perf(R), R)$ extends to a fully faithful functor $\mathrm{Alg}_{\E_{1}}(\Mod(k)) \to \mathrm{Alg}_{\E_{0}}(\Cat^{k}_{\infty})$ with essential image given by the $k$-linear $\infty$-categories equipped with a generator. Indeed, this follows from \cite[Corollary 5.1.2.6, Proposition 4.8.5.8]{LurieHA} as in the proof of Proposition 7.1.2.6 there.

If \eqref{diag:ring-spectra2} is a pullback diagram of $k$-algebras, then all arguments in the proof of \cref{main theorem} can be made after replacing $\Cat_\infty^\perf$ by $\Cat_\infty^k$. In particular, $A'\wtimes{A}{B'} B$ then carries a canonical $k$-algebra structure and \cref{main theorem} holds for any localizing invariant of $k$-linear $\infty$-categories.
\end{rem}

\section{Applications to \textit{K}-theory}

\subsection{Preliminaries}

\begin{dfn}\label{connectivity}
A map $f\colon X \to Y$ of spectra is said to be \emph{$n$-connective} if its fibre $F$ is $n$-connective.\footnote{
We follow the convention of \cite{LurieHTT} which is made such that a spectrum $X$ is as connective as the canonical morphism $X \to 0$. 
Warning: An $n$-connective map in our sense is often called $n$-connected.
} If $\Lambda$ is an abelian group, then $f$ is said to be \emph{$\Lambda$-$n$-connective} if $F\otimes M\Lambda$ is $n$-connective, where $M\Lambda$ is the Moore spectrum for $\Lambda$.
A commutative diagram
\[\begin{tikzcd}
	X \ar[r] \ar[d] & Y \ar[d] \\ Z \ar[r] & W
\end{tikzcd}\]
is said to be \emph{$n$-cartesian} if the canonical map $X \to Z \times_{W} Y$ is $n$-connective, or, equivalently, if the induced map on horizontal fibres is $n$-connective.
Similarly, the above diagram is said to be \emph{$\Lambda$-$n$-cartesian}, if the diagram obtained by tensoring with $M\Lambda$ is $n$-cartesian, or, equivalently, if the canonical map $X \to Z \times_{W} Y$ is $\Lambda$-$n$-connective.
\end{dfn}

\begin{rem}
	\label{rem:connectivity-pullbacks-1}
From the equivalence
\[ 
\Omega\fib( Z \sqcup_{X} Y \to W) \simeq \fib(X \to Z \times_{W} Y) 
\]
we deduce that a diagram is $\Lambda$-$n$-cartesian if and only if the canonical map $Z \sqcup_{X} Y \to W$ is $\Lambda$-$(n+1)$-connective.
\end{rem}

\begin{rem}
	\label{rem:connectivity-pullbacks}
An $n$-cartesian diagram as above in particular gives rise to a long exact Mayer--Vietoris  sequence
\[ 
\pi_{n}(X) \lto \pi_{n}(Y)\oplus\pi_{n}(Z) \lto \pi_{n}(W) \lto \pi_{n-1}(X) \lto \dots
\]
\end{rem}

For future reference, we record the following well-known statement, see \cite[Proposition 1.1]{Waldhausen} for a similar result.
\begin{lemma}\label{lemma:waldhausen-E1}
Let $\Lambda$ be a localisation of $\Z$ or $\Z/m$ for some integer $m$.
Assume that $A \to B$ is a map of connective $\E_1$-rings which induces an isomorphism $\pi_{0}(A) \xrightarrow{\sim} \pi_{0}(B)$ and  is $\Lambda$-$n$-connective for some  $n\geq 1$.
Then the induced map 
\[K(A) \lto K(B) \]
is $\Lambda$-$(n+1)$-connective. In other words, the map $K(A;\Lambda) \to K(B;\Lambda)$ of $K$-theories with coefficients in $\Lambda$ is $(n+1)$-connective.
\end{lemma}
\begin{proof}
As non-positive $K$-groups of a connective $\E_{1}$-ring $A$ only depend on $\pi_0(A)$, see \cite[Theorem 9.53]{BGT}, it suffices to prove the result for connected $K$-theory.
For this we recall the plus-construction of algebraic $K$-theory. For any connective $\E_1$-ring $A$, there is a group-like $\E_1$-monoid $\GL(A)$ which satisfies
\[ \Omega^\infty_0(K(A)) \simeq \BGL(A)^+\]
where $(-)^{+}$ denotes Quillen's plus-construction and the subscript $0$ denotes the component of the trivial element of $K_0(A)$, see e.g.\ \cite[Lemma 9.39]{BGT}.
By construction of $\GL$ and the assumption that $A \to B$ induces an isomorphism on $\pi_{0}$, there is an equivalence
\[
\fib( \GL(A) \to \GL(B) ) \overset{\simeq}{\lto} \fib( \M(A) \to \M(B))
\]
where $\M$ denotes the $\E_{\infty}$-space of matrices. This implies
that the map of spaces
\[ \BGL(A) \lto \BGL(B) \] 
is an isomorphism on fundamental groups and $\Lambda$-$(n+1)$-connective, see for instance \cite{Neisendorfer} for the notion of homotopy groups with coefficients for spaces. It follows that the fibre $F$ of this map is a nilpotent space (we may pass to universal covers without changing the fibre of the map) with abelian fundamental group which is in addition $\Lambda$-$(n+1)$-connective, i.e.\ that $\pi_i(F;\Lambda) = 0$ for $i \leq n$.
We claim that also $\tilde H_{i}(F; \Lambda) = 0$ for $i \leq n$.
 If $\Lambda$ is a localisation of $\Z$, this follows by Serre class theory --- use the Serre class of abelian groups which vanish after tensoring with $\Lambda$, and note that $\pi_i(F;\Lambda) \cong \pi_i(F)\otimes \Lambda$ and $H_i(F;\Lambda) \cong H_i(F;\Z)\otimes \Lambda$ by  flatness. Now consider the case  $\Lambda = \Z/m$. If $n = 1$, the claim follows  from the Hurewicz isomorphism $\pi_1(F;\Lambda) \cong \pi_1(F)\otimes \Lambda \cong H_1(F;\Lambda)$. If $n \geq 2$, then 
$\Tor_1^\Z(\pi_1(F),\Z/m) = 0$ as it is a quotient of $\pi_2(F;\Z/m)$. So we may apply \cite[Theorem 9.7]{Neisendorfer} and deduce the claim.
It follows that the map 
\[ H_i(\BGL(A);\Lambda) \lto H_i(\BGL(B);\Lambda) \]
is an isomorphism for $i \leq n$ and a surjection for $i = n+1$. Since the plus-construction is a homology equivalence it follows that also the map
\[ H_i(\BGL(A)^+;\Lambda) \lto H_i(\BGL(B)^+;\Lambda) \]
is an isomorphism for $i \leq n$ and a surjection for $i=n+1$. Hence the map $\BGL(A)^+ \to \BGL(B)^+$ is  $\Lambda$-$(n+1)$-connective: For $\Lambda$ a localization of $\Z$ this again follows by Serre class theory and for $\Lambda = \Z/m$ from \cite[Corollary 9.15]{Neisendorfer}.
\end{proof}

\subsection{Excision results in algebraic \textit{K}-theory}
In this section, we will apply our main theorem to localizing invariants which satisfy a certain connectivity assumption. We thank Thomas Nikolaus for suggesting to introduce the following definition. As earlier, let $\Lambda$ be a localisation of $\Z$ or $\Z/m$ for some integer $m$.
\begin{dfn}
A spectrum valued localizing invariant $E$ is said to be $\Lambda$-$k$-\emph{connective} if for every map $A \to B$ of connective $\E_1$-rings which is a $\pi_0$-isomorphism and $\Lambda$-$n$-connective for some $n\geq 1$ the induced map $E(A) \to E(B)$ is $\Lambda$-$(n+k)$-connective.
A $\Z$-$k$-connective localizing invariant is simply called $k$-\emph{connective}.
\end{dfn}
Notice  that $\Lambda$-$k$-connective localizing invariants are closed under extensions, i.e.\ if 
\( E' \to E \to E''\)
is a fibre sequence of localizing invariants where $E'$ and $E''$ are $\Lambda$-$k$-connective, then so is $E$.

\begin{ex}
Many  localizing invariants  are $k$-connective for some $k$. In the following examples, $\Lambda$ is as before a localization of $\Z$ or $\Z/m$.
\begin{enumerate}
\item $K$-theory is $\Lambda$-$1$-connective by \cref{lemma:waldhausen-E1},
\item topological Hochschild homology $\THH$ is $\Lambda$-$0$-connective,
\item truncating invariants (see Defininition~\ref{def:truncating-nil-excision}) are $k$-connective for every $k$,
\item topological cyclic homology $\TC$ and rational negative cyclic homology $\HN(-\otimes \Q/\Q)$ are $1$-connective, as each sits in a fibre sequence with $K$-theory and a truncating invariant.
\end{enumerate}

\end{ex}
With this notion we have the following consequence of \cref{main theorem}.

\begin{thm}\label{thm:Suslin-general}
Assume that \eqref{diag:ring-spectra2} is a pullback square of $\E_1$-ring spectra all of which are connective. If the map $A' \otimes_A B \to B'$ is a $\pi_0$-isomorphism and $\Lambda$-$n$-connective for some $n \geq 1$, and if $E$ is a $\Lambda$-$k$-connective localizing invariant, then the diagram
\[\begin{tikzcd}
	E(A) \ar[r] \ar[d] & E(B) \ar[d] \\ E(A') \ar[r] & E(B')
\end{tikzcd}\]
is $\Lambda$-$(n+k-1)$-cartesian.
\end{thm}
\begin{proof}
By \cref{main theorem} and Remark~\ref{rem:connectivity-pullbacks-1} we need to show that the map $E(A'\wtimes{A}{B'} B) \to E(B')$ is $\Lambda$-$(n+k)$-connective. This follows directly from the assumptions.
\end{proof}

As $K$-theory is $\Lambda$-$1$-connective we obtain the following theorem.

\begin{thm}\label{thm:Suslin}
Assume that \eqref{diag:ring-spectra2} is a pullback square of $\E_1$-ring spectra all of which are connective. If the map $A' \otimes_A B \to B'$ is a $\pi_0$-isomorphism and $\Lambda$-$n$-connective for some $n \geq 1$, then the diagram
\[\begin{tikzcd}
	K(A) \ar[r] \ar[d] & K(B) \ar[d] \\ K(A') \ar[r] & K(B')
\end{tikzcd}\]
is $\Lambda$-$n$-cartesian.
\end{thm}

\begin{ex}
	\label{ex:LT-squares}
To illustrate this theorem, consider the case of a diagram \eqref{diag:ring-spectra2} consisting of discrete rings. In all of the following cases the diagram \eqref{diag:ring-spectra2} is a pullback square of $\E_{1}$-ring spectra and the map $A' \otimes_{A} B \to B'$ is a $\pi_{0}$-isomorphism. 
\begin{enumerate}
\item  The diagram \eqref{diag:ring-spectra2} is a \emph{Milnor square}, i.e.~the maps $A \to A'$ and $B \to B'$ are surjective and the map $A \to B$ sends the ideal $I = \ker(A \to A')$ isomorphically to $\ker(B \to B')$.  We have:
\[
\pi_{0}(A' \otimes_{A} B) \cong \Tor_{0}^{A}(A/I, B) \cong B/IB \cong B'.
\]
Since $B'$ is discrete, the map $A' \otimes_{A} B \to B'$ is always  1-connective, and we deduce the classical Mayer-Vietoris sequence of a Milnor square (see \cite[Theorem~XII.8.3]{Bass})
\[ 
\begin{tikzcd}[column sep=small]
\phantom{x}	& K_1(A) \ar[r] & K_1(A')\oplus K_1(B) \ar[r] & K_1(B') \ar[r,"\partial"] & K_0(A) \ar[r] & K_0(A')\oplus K_0(B) \ar[r] & \dots
\end{tikzcd}
\]

\item The map $A \to B$ is an \emph{analytic isomorphism} along a multiplicatively closed set $S$ of central elements of $A$, i.e.~it maps $S$ to central elements of $B$ and induces  isomorphisms on the kernels\footnote{Often, this condition is replaced by the stronger condition that $S$ consists of nonzerodivisors which map to nonzerodivisors in $B$.} and  cokernels of multiplication by $s$ for every $s \in S$, and we have $A' = S^{-1}A$, $B' = S^{-1}B$. These conditions imply that the induced maps $\ker(A \to S^{-1}A) \to \ker(B \to S^{-1}B)$ and $\mathrm{coker}(A \to S^{-1}A) \to \mathrm{coker}(B \to S^{-1}B)$ are also isomorphisms and hence that the square \eqref{diag:ring-spectra2} is a pullback of $\E_1$-ring spectra.
Moreover, $A \to A'$ is flat and thus $A' \otimes_{A} B \to B'$  an equivalence, and in particular $n$-connective for any $n$. We deduce the classical long exact sequence of an analytic isomorphism due to Karoubi, Quillen, and Vorst \cite[Proposition 1.5]{Vorst}
\[
\begin{tikzcd}[column sep =small]
\dots \ar[r] & K_{i}(A) \ar[r] & K_{i}(A') \oplus K_{i}(B) \ar[r] & K_{i}(B') \ar[r,"\partial"] & K_{i-1}(A) \ar[r] & \dots
\end{tikzcd}
\]

\item The diagram \eqref{diag:ring-spectra2} is an affine \emph{Nisnevich square}, i.e.~all rings in it are commutative, the map $\Spec(A') \to \Spec(A)$ is an open immersion, $B' \cong A' \otimes_{A} B$, and $A \to B$ is \'etale  and induces an isomorphism on the closed complements $(\Spec(B) \setminus \Spec(B')) \xrightarrow{\sim} (\Spec(A) \setminus \Spec(A'))$ with the reduced subscheme structure. 
To see that such a Nisnevich square is a pullback square of $\E_1$-ring spectra, one can use 
the Mayer--Vietoris sequence of \'etale cohomology groups 
\[
0 \lto A \lto A' \oplus B \lto B' \lto H^1_{\acute{\mathrm{e}}\mathrm{t}}(\Spec(A);\mathcal{O}) \lto \dots
\]
(see  \cite[Proposition III.1.27]{Milne}) where the last group vanishes as $\Spec(A)$ is affine.
Furthermore, all maps in the diagram are flat, hence $B' \simeq A' \wtimes{A}{B'} B$, and we deduce the long exact Nisnevich Mayer--Vietoris sequence \cite[Theorem 10.8]{ThomasonTrobaugh}.
\end{enumerate}
\end{ex}

If \eqref{diag:ring-spectra2} is a Milnor square of discrete rings (see Example~\ref{ex:LT-squares}(i)), 
 the connectivity condition  in \cref{thm:Suslin} may be expressed as a condition on Tor-groups as in the following corollary.
 For simplicity, we formulate the result only for $K$-theory;   for an arbitrary $k$-connective localizing invariants the resulting square is $\Lambda$-$(n+k-1)$-cartesian.

\begin{cor}\label{cor:Suslin}
Assume that \eqref{diag:ring-spectra2} is a Milnor square of discrete rings. 
\begin{enumerate}
\item If $\Lambda$ is a localisation of $\Z$ and $\Tor_i^A(A',B\otimes_\Z \Lambda) = 0$ for $i = 1,\dots,n-1$, or
\item if $\Lambda = \Z/m$ and the $m$-multiplication on $\Tor_{i}^{A}(A',B)$ is an isomorphism for $i = 1, \dots, n-2$ and surjective for $i = n-1$, 
\end{enumerate}
then the diagram 
\[\begin{tikzcd}
	K(A) \ar[r] \ar[d] & K(B) \ar[d] \\ K(A') \ar[r] & K(B')
\end{tikzcd}\]
is $\Lambda$-$n$-cartesian.
\end{cor}

\begin{rem}
Assume that \eqref{diag:ring-spectra2} is a Milnor square. If $\Tor_{i}^{A}(A',B) = 0$ for $ i = 1, \dots, n-1$, then Corollary \ref{cor:Suslin} says that the induced map on relative $K$-groups (see Definition~\ref{def:rel-K})
\[
K_{i}(A,A') \lto K_{i}(B, B')
\]
is bijective for $i < n$ and surjective for $i=n$. In fact, one can say a little bit more: It follows from Theorem~\ref{main theorem} and \cite[Proposition 1.2]{Waldhausen} that there is an exact sequenc
\[
K_{n+1}(A,A') \lto K_{n+1}(B,B') 
\lto \Tor^{A}_{n}(A', B)/(bx-xb) \lto K_{n}(A,A')  \lto K_{n}(B,B') \lto 0
\]
where the middle term denotes the symmetrization of the $B'$-bimodule $\Tor^{A}_{n}(A',B)$. For example, if $A$ and $B$ are commutative and $I=\ker(A \to A')$, then the symmetrization of $\Tor^{A}_{1}(A',B)$ is isomorphic to $\Omega_{B/A} \otimes_{B} I/I^{2}$ where $\Omega_{B/A}$ denotes the $B$-module of K\"ahler differentials. For $n=1$, the  part of the above exact sequence starting with the Tor-term was obtained by Swan \cite[Corollary 4.7]{Swan} (and Vorst, see \cite[Exercise III.2.6]{Weibel}).
\end{rem}

Let us now explain how this implies Suslin's excision result for non-unital rings. Let $I$ be a not necessarily unital ring and  denote by $\Z \ltimes I$ its unitalisation. It admits a natural augmentation $\Z\ltimes I \to \Z$ and we set
\[ K(I) = \fib(K(\Z\ltimes I) \to K(\Z)).\]

\begin{dfn}[Suslin]
A non-unital ring $I$ is said to satisfy 
 excision in $K$-theory with coefficients in $\Lambda$ up to degree $n$ if for any unital ring $B$ containing $I$ as a two-sided ideal, the diagram
\[ \begin{tikzcd}
	K(\Z\ltimes I) \ar[r] \ar[d] & K(B) \ar[d] \\
	 K(\Z) \ar[r] & K(B/I)
\end{tikzcd}\]
is $\Lambda$-$(n+1)$-cartesian. 
\end{dfn}

We find the following consequence of \cref{thm:Suslin}, which is  \cite[Theorem A]{Suslin}. Here, as before,  $\Lambda$ is a localisation of $\Z$ or $\Z/m$ for some integer $m$.

\begin{thm}
	\label{thm:original-suslin}
The non-unital ring   $I$ satisfies excision in $K$-theory with coefficients in $\Lambda$ up to degree $n$ if $\Tor_i^{\Z\ltimes I}(\Z,\Lambda) = 0$ for $i = 1, \dots, n$.
\end{thm}
\begin{proof}
Let $B$ be any ring containing $I$ as a two-sided ideal.
We observe that the diagram
\begin{equation}
	\label{diag:Milnor-Suslin}
\begin{tikzcd}
	\Z \ltimes I \ar[r] \ar[d] & B \ar[d] \\
	\Z \ar[r] & B/I
\end{tikzcd}
\end{equation}
is a Milnor square. The assumption implies that the multiplication map $\Z \otimes_{\Z\ltimes I} \Z \to \Z$ is $\Lambda$-$(n+1)$-connective, since $\Z \otimes M\Lambda \simeq \Lambda$.
Lemma~\ref{lemma Tor-unitality} below together with the fact that $B/I$ is discrete imply that $\Z \otimes_{\Z \ltimes I} B \to B/I$ is $\Lambda$-$(n+1)$-connective. Obviously, it is also a $\pi_{0}$-isomorphism. Now Theorem~\ref{thm:Suslin} implies that the diagram of $K$-theory spectra induced from the Milnor square \eqref{diag:Milnor-Suslin} is $\Lambda$-$(n+1)$-cartesian, as desired.
\end{proof}

\begin{lemma}\label{lemma Tor-unitality}
Suppose that \eqref{diag:ring-spectra2} is a pullback square of  $\E_1$-rings all of which are connective. If the map $A'\otimes_A A' \to A'$ is $\Lambda$-$n$-connective, then the map $A'\otimes_A B \to B'$ is $\Lambda$-$(n-1)$-connective and in addition induces an isomorphism on $\pi_{n-1}$ with coefficients in $\Lambda$. If $n= \infty$, then the same conclusion holds for possibly non-connective $\E_1$-rings.
\end{lemma}
\begin{proof}
Since $A'\otimes_A A' \to A'$ is $\Lambda$-$n$-connective, the same holds for the map $A' \otimes_A B' \to B'$ obtained by base change along $A' \to B'$. It thus suffices to show that after tensoring with $M\Lambda$ the map $A'\otimes_A B \to A'\otimes_A B'$ is $(n-1)$-connective and induces an isomorphism on $\pi_{n-1}$. 
For this we consider the pullback square 
\[
\begin{tikzcd}
 A' \otimes M\Lambda \ar[d]\ar[r] & A' \otimes_{A} B \otimes M\Lambda \ar[d] \\ 
 A' \otimes_{A} A' \otimes M\Lambda \ar[r] & A' \otimes_{A} B' \otimes M\Lambda 
\end{tikzcd}
\]
obtained from \eqref{diag:ring-spectra2} by tensoring with $A'$ over $A$ and with $M\Lambda$.
The left vertical map is split by the multiplication map. By assumption, this split is $n$-connective. A little diagram chase using the long exact sequences of homotopy groups associated to the vertical maps now gives the claim.
The assertion about non-connective $\E_{1}$-rings follows analogously.
\end{proof}

\begin{rem}
It is worthwhile to clarify the relation between the various Tor-vanishing conditions. Suppose \eqref{diag:ring-spectra2} is a Milnor square. By \cref{cor:Suslin}, excision in $K$-theory holds for this particular square provided $\Tor_i^A(A',B) = 0$ for $i \geq 1$. Suslin proved that excision holds for \emph{all} Milnor squares with fixed ideal $I$ if $\Tor_i^{\Z\ltimes I}(\Z,\Z) = 0$ for $i\geq 1$. 
Morrow showed in \cite[Theorem 0.2]{Morrow-pro-unitality}  that Suslin's condition is in fact equivalent to the vanishing of $\Tor_i^A(A',A')$ for $i\geq 1$ and \emph{some} Milnor square \eqref{diag:ring-spectra2} with ideal $I$. 
Lemma~\ref{lemma Tor-unitality} and the following example show  that the condition of \cref{cor:Suslin} is the most general of the above.
\end{rem}

\begin{ex} 
Let $A'$ be a commutative ring, and $f \in A'$. We let $A'_f = A'[1/f]$ be the localization and consider the Milnor square
\[ \begin{tikzcd}
	A \ar[r] \ar[d] & A'_f[x] \ar[d] \\ 
	A' \ar[r] & A'_f 
\end{tikzcd}\]
where the right vertical map sends the variable $x$ to zero. On the one hand, the map $A \to A'_f[x]$ is a localization and thus flat. On the other hand $\Tor_1^A(A',A') \cong I/I^2$, where $I = \ker(A \to A')$, and the variable $x$  obviously gives rise to a non-zero element in $I/I^2$.
\end{ex}

\subsection{Torsion in birelative and relative \textit{K}-groups}

In this section we explain how to use our main theorem to prove stronger versions of the results obtained by Geisser--Hesselholt in \cite{GH2}. We thank Akhil Mathew for pointing out that our main theorem should give a direct proof of these results.   

\begin{dfn}
	\label{def:rel-K}
The \emph{relative $K$-theory} of a map of $\E_{1}$-rings $A \to B$ is defined as the fibre of the induced map of $K$-theory spectra,
\[
K(A,B) = \fib( K(A) \to K(B) ),
\]
its homotopy groups are denoted by $K_{i}(A,B)$. Similarly, if \eqref{diag:ring-spectra2} is a commutative square of $\E_{1}$-rings, the \emph{birelative $K$-theory} is 
\[
K(A,B,A',B') = \fib( K(A,B) \to K(A',B') ).
\]
\end{dfn}

The following is an immediate consequence of \cref{main theorem}.
\begin{lemma}\label{birelative-term}
If the commutative square \eqref{diag:ring-spectra2} is a pullback square of $\E_1$-rings, then there is a canonical equivalence
\[ 
K(A,B,A',B') \simeq \Omega \fib( K(A'\wtimes{A}{B'} B) \to K(B') ).
\] 
\end{lemma}

Fix an integer $N \geq 1$. We say that an abelian group $A$ is a bounded $N$-torsion group if $A$ is killed by a power of $N$. 

\begin{prop}\label{prop p torsion K groups}
Let $A \to B$ be a 1-connective map of connective $\E_1$-ring spectra. Assume that $\pi_{i}(\fib(A\to B))$ is a bounded $N$-torsion group for $i \leq n$.
 Then  the relative K-group $K_{i}(A,B)$
is a bounded $N$-torsion group for each $i \leq n+1$.
\end{prop}
\begin{proof}
The proof is  similar to that of Lemma~\ref{lemma:waldhausen-E1}. Since the map $K(A) \to K(B)$ is 2-connective, we may again use the plus-construction of $\BGL$ to compute the relative K-groups.
Consider the fibre sequences
\begin{align*}
&F \lto  \BGL(A) \lto \BGL(B), \\
&\tilde{F} \lto \BGL(A)^+ \lto \BGL(B)^+ ,
\end{align*}
and note that $K_i(A,B) \cong \pi_i(\tilde{F})$.
The assumptions imply that $F$ is simply connected, and that $\pi_{i}(F)$ is a bounded $N$-torsion group for $i \leq n+1$. As the bounded $N$-torsion   groups  form a Serre class of abelian groups, it follows that the reduced homology $\tilde{H}_i(F;\Z)$ is a bounded $N$-torsion group for every $i \leq n+1$. The relative Serre spectral sequence 
\[ 
E^2_{p,q} = H_p(\BGL(B);\tilde{H}_q(F;\Z)) \Longrightarrow H_{p+q}(\BGL(A),\BGL(B);\Z) 
\]
then implies  that $H_{i}(\BGL(A),\BGL(B);\Z)$  is a bounded $N$-torsion group for $i \leq n+1$. Now consider the relative Serre spectral sequence for the fibration of plus-constructions. Using that the map to the plus construction is acyclic we find a spectral sequence
\[ 
E^2_{p,q} = H_p(\BGL(B);\tilde{H}_{q}(\tilde{F};\Z)) \Longrightarrow H_{p+q}(\BGL(A),\BGL(B);\Z).
\]
whose abutment is a bounded $N$-torsion group in degrees $\leq n+1$ by the above. Since this Serre spectral sequence is untwisted, we see by induction over $i$ that also $\tilde{H}_i(\tilde{F};\Z)$ is a bounded $N$-torsion group for $i \leq n+1$. As $\tilde{F}$ is also simply connected, again by Serre class theory it follows that also $\pi_i(\tilde{F})$ for $i \leq n+1$ are bounded $N$-torsion groups.
\end{proof}

\begin{thm}
Assume that \eqref{diag:ring-spectra2} is a pullback square of $\E_1$-ring spectra all of which are connective. If the map $A' \otimes_A B \to B'$ is 1-connective, and if the homotopy groups of its fibre are  bounded $N$-torsion groups in degrees $\leq n$, then the birelative $K$-groups $K_{i}(A,B,A',B')$ are bounded $N$-torsion groups for every $i \leq n$.
\end{thm}
\begin{proof}
The claim follows directly from the assumptions, \cref{birelative-term}, and  \cref{prop p torsion K groups}.
\end{proof}

In the case of a Milnor square, we may again express the assumptions in the previous theorem in terms of Tor-groups as follows.
\begin{cor}
	\label{cor:GH-birelative-torsion}
Assume that \eqref{diag:ring-spectra2} is a Milnor square of discrete rings such that $\Tor_i^A(A',B)$ is a bounded $N$-torsion group for all $i = 1, \dots, n$. Then the birelative K-group $K_{i}(A,B,A', B')$ is a bounded $N$-torsion group for every $i \leq n$.
\end{cor}

\begin{proof}
As all rings are discrete, $\pi_i$ of the fibre of the map $A'\otimes_A B \to B'$ is given by $\Tor_i^A(A',B)$ for $i\geq 1$ and is trivial otherwise. The claim thus follows from the previous theorem.
\end{proof}

\begin{rem}
The assumptions of the corollary are automatically satisfied for any integer $n$ if $A'$ or $B$ is a $\Z/N$-algebra. If $A$ itself is a $\Z/N$-algebra, then the corollary is due to Geisser--Hesselholt \cite[Theorem C]{GH2}, see also \cref{second remark about GH}.
\end{rem}

Here is an example where Theorem C of \cite{GH2} cannot be applied, but \cref{cor:GH-birelative-torsion} does apply.
Let $G$ be a finite group of order $N$. Let $\mathfrak{M}$ be a maximal $\Z$-order of the rational group ring $\Q[G]$ containing $\Z[G]$.
Then $I = N\cdot \mathfrak{M}$ is a common ideal in $\Z[G]$ and $\mathfrak{M}$ \cite[Corollary XI.1.2]{Bass} so that we get a Milnor square
\[
\begin{tikzcd}
 \Z[G] \ar[d]\ar[r] & \mathfrak{M} \ar[d] \\ 
 \Z[G]/I \ar[r] & \mathfrak{M}/I .
\end{tikzcd}
\]

\begin{cor}
For a finite group $G$ the birelative $K$-groups $K_{*}(\Z[G], \mathfrak{M}, \Z[G]/I, \mathfrak{M}/I)$ are bounded $N$-torsion groups, and for $i \geq 1$, the relative $K$-groups $K_i(\Z[G],\mathfrak{M})$ are torsion groups of bounded exponent.
\end{cor}
\begin{proof}
Since $\Z[G]/I$ is a $\Z/N$-algebra, Corollary~\ref{cor:GH-birelative-torsion} implies the statement about birelative $K$-groups. Since $\Z[G]/I$ and $\mathfrak{M}/I$ are both finite, their $K$-groups in positive degrees are also finite by a result of Kuku \cite[Proposition IV.1.16]{Weibel}. Hence the relative $K$-groups $K_{i}( \Z[G]/I, \mathfrak{M}/I)$ are finite for $i \geq 1$. This implies the claim. 
\end{proof}

\begin{rem}
From the localization sequence for $\mathfrak{M}$ \cite[Theorem 1.17]{Oliver} we see that also the relative $K$-groups $K_{i}(\mathfrak{M}, \Q[G])$ are torsion for $i \geq 1$. It follows that 
$K_{i}(\Z[G]) \otimes \Q \to K_{i}(\Q[G]) \otimes \Q$ is an isomorphism for $i \geq 2$, a result due to Borel \cite[Theorem IV.1.17]{Weibel}, and injective for $i=1$. 
\end{rem}

The following is Theorem~\ref{intro-thm:D} from the introduction. It may be proved directly using a Serre spectral sequence argument, similarly as in the proof of Proposition~\ref{prop p torsion K groups}. Here we give an alternative argument using Proposition~\ref{prop p torsion K groups} to reduce it to the birelative case.

\begin{thm}
	\label{thm:torsion-relative-K-groups}
Let $A$ be a discrete ring, and let $I$ be a two-sided nilpotent ideal in $A$. Assume that, as an abelian group, $I$ is a bounded $N$-torsion group. Then the relative K-groups $K_{*}(A,A/I)$ are bounded $N$-torsion groups.
\end{thm}
\begin{proof}
Let $n$ be the smallest natural number such that $I^{n} = 0$. Consider the two projections $A \to A/I^2 \to A/I$ and the associated fibre sequence of relative $K$-theory spectra
\[ K(A,A/I^2) \lto K(A,A/I) \lto K(A/I^2,A/I).\]
By induction on $n$ we may assume that the claim holds for $K(A,A/I^2)$. Since the kernel of $A/I^2 \to A/I$ is a square zero ideal, it hence suffices to treat the case $n=2$.

In this case $I$ is canonically an $A/I$-bimodule, so we may form the two differential graded algebras
\begin{equation}
	\label{eq:CIA}
C(I,A) = [I \stackrel{i}{\lto} A ] \quad \text{and}\quad C(I,A/I) = [ I \stackrel{0}{\lto} A/I ]
\end{equation}
concentrated in degrees 1 and 0. Note that $C(I,A) \simeq A/I$.
There is a canonical map $C(I,A) \to C(I,A/I)$ given by the identity on $I$ and the projection on $A$. We now consider the commutative diagram
\begin{equation}
	\label{diag:CIA-pullback}
\begin{tikzcd}
	A \ar[r] \ar[d] & A/I \ar[d] \\ 
	C(I,A) \ar[r] & C(I,A/I) 
\end{tikzcd}
\end{equation}
in which the vertical maps are the canonical inclusions in degree zero. 
As the induced map on horizontal fibres is an equivalence, this is a pullback diagram. For ease of notation let us also write $B' = C(I,A/I)$.

We claim that the map $C(I,A) \wtimes{A}{B'} A/I \lto C(I,A/I)$ is a 1-truncation. Indeed, by Remark~\ref{rem:multiplication-cofibre-sequence} this map sits in a diagram of cofibre sequences
\[ 
\begin{tikzcd}
	I \otimes_A A/I \ar[r] \ar[d] & A/I \ar[r] \ar[d, equal] & C(I,A) \wtimes{A}{B'} A/I \ar[d] \\
	I \ar[r] & A/I \ar[r] & C(I,A/I)
\end{tikzcd}
\]
in which the left vertical map is clearly surjective on $\pi_{1}$ and an isomorphism on $\pi_0$ by the assumption that $I^2=0$.
Hence, the map in question is an isomorphism on $\pi_0$ and $\pi_1$ by the 5-lemma.

For $i \geq 2$, the homotopy groups of the fibre of the map $C(I,A) \wtimes{A}{B'} A/I \lto C(I,A/I)$ are  given by
\[
\pi_{i} ( C(I,A) \wtimes{A}{B'} A/I ) \cong \Tor^{A}_{i}(A/I, A/I) \cong \Tor^{A}_{i-1}(I, A/I).
\]
These are bounded $N$-torsion groups by the assumption on $I$.
Hence the fibre of the map
$
K(C(I,A)\wtimes{A}{B'} A/I) \to K(C(I,A/I))
 $
has bounded $N$-torsion homotopy groups by \cref{prop p torsion K groups}.
Since by Lemma~\ref{birelative-term} this fibre is the birelative $K$-theory of  diagram \eqref{diag:CIA-pullback}, it  is enough to show that the fibre of the map $K(C(I,A)) \to K(C(I,A/I))$
has bounded $N$-torsion homotopy groups. As the composition
$C(I,A) \to C(I,A/I) \to A/I $
is an equivalence, it suffices to argue that the fibre of $K(C(I,A/I)) \to K(A/I)$
has bounded $N$-torsion homotopy groups. This follows again from \cref{prop p torsion K groups} since $C(I,A/I) \to A/I$ is a 0-truncation with fibre $I[1]$ which has bounded $N$-torsion homotopy groups by assumption.
\end{proof}

\begin{rem}\label{second remark about GH}
The conditions of \cref{cor:GH-birelative-torsion} and \cref{thm:torsion-relative-K-groups} are automatically satisfied if $A$ is a $\Z/N$-algebra. In this form, these results are due to Geisser--Hesselholt \cite[Theorem C \& A]{GH2}. Geisser and Hesselholt first prove the relative case through topological cyclic homology. Using the relative case, and a pro-excision result in $K$-theory, which we address in \cref{pro excision GH}, Geisser and Hesselholt deduce the birelative case. 
\end{rem}

\subsection{Pro-excision results}
	\label{sec:pro-excision}

Recall that a pro-object  $\{X_{\lambda} \}_{\lambda\in\Lambda}$, or simply $\{X_{\lambda} \}$, in an $\infty$-category $\cC$ is a diagram in $\cC$ indexed by a small cofiltered $\infty$-category $\Lambda$.
 The pro-objects in $\cC$ form an $\infty$-category $\Pro(\cC)$. Up to equivalence, any map in $\Pro(\cC)$ can be represented by a level map, i.e.~a natural transformation of diagrams.
We are interested in the cases where $\cC$ is  the $\infty$-category $\Sp$  or $\Alg^{\cn}$ of spectra or connective $\E_{1}$-rings in spectra, respectively. 
Recall the forgetful functor $\RMod \to \Alg$ from \cite[Definition~4.2.1.13]{LurieHA} whose fibre over $A$ is the $\infty$-category of $A$-right modules.
We will also consider pro-right modules over a pro-system of connective $\E_{1}$-ring spectra $\{ A_{\lambda}\}$ by which we mean an object of the fibre $\Pro(\RMod)_{\{ A_{\lambda}\}}$ of the canonical functor $\Pro(\RMod) \to \Pro(\Alg)$ over $\{ A_{\lambda} \}$, and similarly for left and bimodules.
For each of these categories of pro-objects  we have level-wise truncation functors $\tau_{\leq n}\colon \Pro(\cC) \to \Pro(\cC)$ and homotopy pro-group functors $\pi_{i}\colon \Pro(\cC) \to \Pro(\mathrm{Ab})$.

\begin{dfn}
A map $f$ in $\Pro(\Sp)$ is called \emph{$n$-connective} if $\tau_{\leq n-1}(\fib(f))$ is zero in $\Pro(\Sp)$; it is called a \emph{weak equivalence} if it is  $n$-connective for all integers $n$. 
A square in $\Pro(\Sp)$ is called \emph{$n$-cartesian} if the map from the upper left corner to the pullback is $n$-connective; it is called \emph{weakly cartesian} if it is $n$-cartesian for every integer $n$.

Any of the $\infty$-categories $\cC$ of above has a canonical forgetful functor $\cC \to \Sp$. We say that a map or a square in $\Pro(\cC)$ is  $n$-connective, a weak equivalence, $n$-cartesian or weakly cartesian, respectively,  if its image in $\Pro(\Sp)$ is  so.
\end{dfn}

For any of the categories considered above, we call a pro-object $\{ X_{\lambda}\}$ bounded above or below if  for some integer $k$ the truncation $\{ \tau_{\leq k}X_{\lambda}\}$ is equivalent to $\{ X_{\lambda}\}$ or to $0$ via the canonical map, respectively. It is called bounded if it is bounded below and above.
A simple induction over the Postnikov tower shows that a map of bounded below objects in $\Pro(\Sp)$ is a weak equivalence if and only if it induces an isomorphism on all homotopy pro-groups.

\begin{lemma}
	\label{lem:weak-equivalences}
\begin{enumerate}
\item For any $\{ A_{\lambda}\} \in \Pro(\Alg^{\cn})$ the forgetful functor $\Pro(\RMod)_{\{ A_{\lambda}\}} \to \Pro(\Sp)$ is conservative on bounded objects.
\item The forgetful functor $\Pro(\Alg^{\cn}) \to \Pro(\Sp)$ is conservative on bounded objects.
\end{enumerate}
\end{lemma}
In particular, in both cases a map between bounded below pro-objects is a weak equivalence if and only if each of its truncations is an equivalence.
\begin{proof}
Let $f\colon \{ M_{\lambda} \} \to \{ N_{\lambda} \}$ be a level map of bounded pro-modules and assume that the underlying map of pro-spectra is an equivalence. As both pro-modules are bounded, we can argue by a finite induction over the Postnikov tower. For the inductive step we use that $\pi_{n}(f)$, which is an isomorphism of pro-abelian groups, is in fact an isomorphism of (discrete) pro-modules over $\{ A_{\lambda}\}$, see for instance the proof of \cite[Lemma~13.2]{MR1828474}. This proves (i).

For (ii) we again argue by induction over the Postnikov tower. For the inductive step, we use that for any connective $\E_{1}$-algebra $A$, the map $\tau_{\leq n}(A) \to \tau_{\leq n-1}(A)$ is a square-zero extension with fibre $\pi_{n}(A)[n]$ by \cite[Corollary 7.4.1.28]{LurieHA} and hence there is a pullback diagram of $\E_{1}$-ring spectra
\[
\begin{tikzcd}
 \tau_{\leq n}(A) \ar[d]\ar[r] & \tau_{\leq n-1}(A) \ar[d] \\ 
 \tau_{\leq n-1}(A) \ar[r] & \tau_{\leq n-1}(A) \oplus \pi_{n}(A)[n+1] 
\end{tikzcd}
\]
where the lower right corner denotes the trivial square-zero extension and depends functorially on the $\tau_{\leq n-1}(A)$-bimodule $\pi_{n}(A)$. This reduces the inductive step to the case of (discrete) bimodules, which follows as above.
\end{proof}

\begin{lemma}\label{KSTlemma}
Let $\{A_\lambda\}$ be a pro-system of connective $\E_{1}$-rings.  
Suppose $f \colon \{ M_\lambda\} \to \{N_\lambda\}$ is a weak equivalence of pro-systems of connective $\{A_\lambda\}$-right modules and let 
$\{B_\lambda\}$ be a pro-system of connective $\{A_{\lambda}\}$-left modules.
 Then also the induced map 
\[ 
\{M_\lambda \otimes_{A_\lambda} B_\lambda\} \lto \{N_\lambda \otimes_{A_\lambda} B_\lambda \} 
\]
is a weak equivalence.
\end{lemma}
\begin{proof}
By Lemma~\ref{lem:weak-equivalences} each truncation $\tau_{\leq n}(f) \colon \{ \tau_{\leq n}(M_\lambda) \} \to \{ \tau_{\leq n}(N_\lambda) \}$
is an equivalence of pro-modules, and hence yields an equivalence $ \{ \tau_{\leq n}(M_\lambda) \otimes_{A_{\lambda}} B_{\lambda} \} \xrightarrow{\sim} \{ \tau_{\leq n}(N_\lambda) \otimes_{A_{\lambda}} B_{\lambda} \}$. Applying $\tau_{\leq n}$ yields the equivalence $\{ \tau_{\leq n}(M_{\lambda} \otimes_{A_{\lambda}} B_{\lambda}) \} \xrightarrow{\sim} \{ \tau_{\leq n}(N_{\lambda} \otimes_{A_{\lambda}} B_{\lambda} )\}$ as required.
\end{proof}

\begin{lemma}
	\label{lem:factorization-pro-connective}
Any $n$-connective map in $\Pro(\Alg^{\cn})$ is equivalent to a level map which is level-wise $n$-connective.
\end{lemma}
\begin{proof}
Any map of pro-objects is equivalent to a level map. So assume given a level map $f \colon \{ A_{\lambda} \} \to \{ B_{\lambda} \}$ which  is $n$-connective.
By Lemma~\ref{lem:weak-equivalences}, the map $\tau_{\leq n-1}(f)$ is an equivalence. Hence, replacing $\{B_{\lambda}\}$ with the equivalent pro-$\E_{1}$-ring $\{ \tau_{\leq n-1}(A_{\lambda}) \times_{\tau_{\leq n-1}(B_{\lambda})} B_{\lambda} \}$, we may assume that $\tau_{\leq n-1}(A_{\lambda}) \to \tau_{\leq n-1}(B_{\lambda})$ is an equivalence for every $\lambda$. 

Let $\overline{\pi}_{n}(B_{\lambda})$ denote the image of the map $\pi_{n}(A_{\lambda}) \to \pi_{n}(B_{\lambda})$.
Let $\tau_{\leq n-1}(B_{\lambda}) \to \pi_{n}(B_{\lambda})[n+1]$ be the derivation that classifies the square-zero extension $\tau_{\leq n}(B_{\lambda}) \to \tau_{\leq n-1}(B_{\lambda})$. Since $\tau_{\leq n-1}(A_{\lambda}) \simeq \tau_{\leq n-1}(B_{\lambda})$, this derivation factors through a derivation $\tau_{\leq n-1}(B_{\lambda}) \to \ol{\pi}_{n}(B_{\lambda})[n+1]$ and we write $\ol{B}_{\lambda} \to \tau_{\leq n-1}(B_{\lambda})$ for the square-zero extension classified by the latter. 
By construction, the map $\tau_{\leq n}(A_{\lambda}) \to \tau_{\leq n}(B_{\lambda})$ factors through an $n$-connective map $\tau_{\leq n}(A_{\lambda}) \to \ol{B}_{\lambda}$ for every $\lambda$. Furthermore the map $\{\ol{B}_\lambda\} \to \{\tau_{\leq n}(B_\lambda)\}$ is an equivalence by the assumption that $f$ is  $n$-connective. 
Finally, the map $f$ factors as  $\{ A_{\lambda} \} \to \{ \ol{B}_{\lambda} \times_{\tau_{\leq n}(B_{\lambda})} B_{\lambda} \} \to \{ B_{\lambda} \}$ where the first map is
 level-wise $n$-connective and the second one is an equivalence.
\end{proof}

\begin{cor}
	\label{cor:connectivity-pro}
	Let $E$ be a $k$-connective localizing invariant. If $f$ is an $n$-connective map of pro-systems of connective $\E_{1}$-ring spectra, then $E(f)$ is $(n+k)
$-connective.  In particular, $E$ preserves weak equivalences.
\end{cor}
\begin{proof}
This follows immediately from Lemma~\ref{lem:factorization-pro-connective}.
\end{proof}

The following results are formulated for algebraic $K$-theory. They admit obvious generalizations for $k$-connective localizing invariants.
We  consider a pro-system of commutative diagrams of connective $\E_{1}$-rings indexed by $\Lambda$ as follows.
\begin{equation}
	\label{diag:pro-spectra}
	\begin{tikzcd}
	A_{\lambda} \ar[r] \ar[d] & B_{\lambda} \ar[d] \\
	A_{\lambda}' \ar[r] & B_{\lambda}'
	\end{tikzcd}
\end{equation}

\begin{thm}
	\label{thm:pro-excision}
Fix an integer $n \geq 1$.
Assume that the diagram of connective pro-$\E_{1}$-rings given by  \eqref{diag:pro-spectra} is weakly cartesian and that
the canonical map $\{  A'_{\lambda} \otimes_{A_{\lambda}} B_{\lambda} \} \to \{ B'_{\lambda}\}$ is weakly $n$-connective.
Then the diagram of pro-spectra
\[
\begin{tikzcd}
\{K(A_{\lambda}) \} \ar[d] \ar[r] & \{ K(B_{\lambda}) \} \ar[d] \\
\{K(A'_{\lambda})\} \ar[r] & \{ K(B'_{\lambda}) \}
\end{tikzcd}
\]
is weakly $n$-cartesian.
\end{thm}

\begin{proof}
We define the $\E_{1}$-ring $C_{\lambda}$ via the pullback
\begin{equation*}
	\label{diag:pro-simplicial}
\begin{tikzcd}
 C_{\lambda} \ar[d]\ar[r] & B_{\lambda} \ar[d] \\ 
 A'_{\lambda} \ar[r] & B'_{\lambda}
\end{tikzcd}
\end{equation*}
By the first assumption 
the canonical map $\{A_{\lambda}\} \to \{ C_{\lambda} \}$ is a weak equivalence of pro-$\E_{1}$-rings.
We claim that also $\{ A'_{\lambda} \otimes_{A_{\lambda}} B_{\lambda} \} \to \{ A'_{\lambda} \otimes_{C_{\lambda}} B_{\lambda} \}$ is a weak equivalence. To see this, we write the source of this map as $\{ A'_{\lambda} \otimes_{A_{\lambda}} C_{\lambda} \otimes_{C_{\lambda}} B_{\lambda} \}$. Under this identification, the map is given by extension of scalars along $\{ C_{\lambda} \} \to \{ B_{\lambda}\} $ of the multiplication map $\{ A'_{\lambda} \otimes_{A_{\lambda}} C_{\lambda} \} \to \{ A'_{\lambda} \}$. By \cref{KSTlemma} it then suffices to show that the latter is a weak equivalence. This map has a section 
given by $\{ A'_{\lambda} \}  \simeq \{ A'_{\lambda} \otimes_{A_{\lambda}}  A_{\lambda} \} \to \{ A'_{\lambda} \otimes_{A_{\lambda}} C_{\lambda} \}$, and as $\{A_{\lambda}\} \to \{ C_\lambda \}$ is a weak equivalence, so is the former by \cref{KSTlemma} again.

We  set $R_{\lambda} = A'_{\lambda} \wtimes{C_{\lambda}}{B'_{\lambda}} B_{\lambda}$. 
Together with the above, the assumption implies that the map of pro-$\E_{1}$-rings $\{R_{\lambda}\} \to \{B'_{\lambda} \}$ is weakly $n$-connective, and hence the induced map on $K$-theory pro-spectra is weakly $(n+1)$-connective by Corollary~\ref{cor:connectivity-pro}. The result now follows by Theorem~\ref{main theorem}.
\end{proof}

In the special case of a Milnor square of discrete rings, the theorem in particular gives the following.

\begin{cor}\label{pro excision GH}
Let  $A \to B$ be a ring homomorphism sending the two-sided ideal $I \subseteq A$ isomorphically onto the two-sided ideal $J \subseteq B$. If the pro-Tor-groups $\{ \Tor_{i}^{A}(A/I^{\lambda}, B) \}_{\lambda\in \N}$ vanish for $i = 1, \dots, n-1$, then  the diagram of pro-spectra
\[
\begin{tikzcd}
K(A) \ar[d] \ar[r] & K(B) \ar[d] \\
\{K(A/I^{\lambda})\} \ar[r] & \{ K(B/J^{\lambda}) \}
\end{tikzcd}
\]
is weakly $n$-cartesian.
\end{cor}
\begin{proof}
Apply Theorem~\ref{thm:pro-excision} with $A_{\lambda} = A$, $A'_{\lambda} = A/I^{\lambda}$, $B_{\lambda} = B$, and $B'_{\lambda} = B/J^{\lambda}$.
\end{proof}

For $n=\infty$ and under the stronger assumption that the  ideal $I$ is (rationally) pro-Tor-unital, i.e.~that the pro-Tor-groups $\{ \Tor_{i}^{\Z \ltimes I^{\lambda}}(\Z,\Z) \}$ vanish (rationally) for $i >0$, the same result was first proven rationally by Corti{\~n}as \cite[Theorem 3.16]{Cortinas} and later integrally by Geisser--Hesselholt \cite[Theorem~3.1]{GH2}, \cite[Theorem~1.1]{GH}. They use pro-versions of the method of Suslin and Wodzicki \cite{Suslin-Wod,Suslin}, which is based on the homology of affine groups.
Their results   can in fact be deduced from ours, similarly as in the proof of Theorem~\ref{thm:original-suslin}. 
Morrow \cite[Theorem 0.2]{Morrow-pro-unitality} proves that the ideal $I$ is pro-Tor-unital if and only if the pro-Tor groups $\{ \Tor^{A}_{i}(A/I^{\lambda}, A/I^{\lambda})\}$ vanish for $i >0$. Our condition is still weaker.

If we moreover assume that $A$ is commutative and noetherian, then the condition of the previous corollary is automatically satisfied.
The following was proven before by Morrow \cite[Corollary 2.4]{Morrow-pro-unitality} by noting that ideals in commutative noetherian rings satisfy the above vanishing condition for pro-Tor-groups and using  Geisser--Hesselholt's pro-excision theorem.

\begin{cor}\label{first cor pro exc}
Let $A$ be a commutative noetherian ring, and let  $A \to B$ be a ring homomorphism sending the ideal $I \subseteq A$ isomorphically onto the two-sided ideal $J \subseteq B$. Then  the diagram of pro-spectra
\[
\begin{tikzcd}
K(A) \ar[d] \ar[r] & K(B) \ar[d] \\
\{K(A/I^{\lambda})\} \ar[r] & \{ K(B/J^{\lambda}) \}
\end{tikzcd}
\]
is weakly cartesian.
\end{cor}
\begin{proof}
Since $A$ is noetherian, the pro-group $\{ \Tor^{A}_{i}(A/I^{\lambda}, M) \}$ vanishes for $i >0$ and any finitely generated $A$-module $M$ \cite[Lemme X.11]{Andre}. This implies that $\{A/I^{\lambda} \otimes_{A} A/I \} \to A/I$ is a weak equivalence. Extending scalars along $A/I \to B/J$ we deduce, using \cref{KSTlemma}, that also $\{ A/I^{\lambda} \otimes_{A} B/J \} \to B/J$ is a weak equivalence. Using the above  vanishing of pro-Tor-groups for $M = J \cong I$ and the long exact sequence of Tor-groups associated with $J \to B \to B/J$ we conclude that $\{ \Tor^{A}_{i}(A/I^{\lambda}, B) \}$ vanishes for every $i > 0$ as desired.
\end{proof}

Similarly, we can derive a pro-excision result for noetherian simplicial commutative rings from Theorem~\ref{thm:pro-excision}. Recall that a simplicial commutative ring $A$ is called noetherian if $\pi_{0}(A)$ is noetherian and each $\pi_{i}(A)$ is a finitely generated $\pi_{0}(A)$-module.
We now fix a commutative ring $R$ and a finite sequence $\bc = (c_{1}, \dots, c_{r})$ of elements in $R$. 
If $A$ is a simplicial commutative $R$-algebra, we set 
\[
A \mmod \bc = A \otimes_{R[x_{1}, \dots, x_{r}]} R
\]
where the ring map $R[x_{1}, \dots, x_{r}] \to R$ sends the variable $x_{i}$ to $c_{i}$, and the tensor product is, as always, derived. If $A$ is discrete, the homotopy groups $\pi_{i}(A \mmod \bc)$ are the Koszul homology groups $H_{i}(A; \bc)$.
For a positive integer $\mu$ we denote by $\bc(\mu)$ the sequence $(c_{1}^{\mu}, \dots, c_{r}^{\mu})$, and for an $R$-module $M$ we write $\bc(\mu)M$ for the submodule $c_{1}^{\mu}M + \dots + c_{r}^{\mu}M$ of $M$.
The following corollary is precisely \cite[Theorem 4.11]{KST} whose proof is based on an adaption of  Suslin's and Wodzicki's method to the noetherian and pro-simplicial setting. In \cite{KST} this is used to prove that $K$-theory satisfies pro-descent for abstract blow-up squares.

\begin{cor}\label{second cor pro exc}
Consider a morphism of pro-systems of noetherian simplicial commutative $R$-algebras $\phi\colon \{A_{\lambda}\} \to \{ B_{\lambda}\}$. Assume that $\phi$ induces an isomorphism 
\[
\{ \bc(\mu) \pi_{i}(A_{\lambda}) \}_{\lambda,\mu} \stackrel{\cong}{\lto} \{ \bc(\mu) \pi_{i}(B_{\lambda}) \}_{\lambda,\mu}
\]
for all $i \geq 0$. Then the diagram of pro-spectra
\[
\begin{tikzcd}
 \{ K(A_{\lambda}) \}_{\lambda} \ar[d]\ar[r] & \{ K(B_{\lambda}) \}_{\lambda} \ar[d] \\ 
 \{ K(A_{\lambda}\mmod \bc(\mu) ) \}_{\lambda,\mu} \ar[r] & \{ K(B_{\lambda}\mmod \bc(\mu) ) \}_{\lambda,\mu} 
\end{tikzcd}
\]
is weakly cartesian.
\end{cor}
\begin{proof}
In order to apply Theorem \ref{thm:pro-excision}, we view $\{A_{\lambda}\}$ as a pro-system indexed by $(\lambda,\mu) \in \Lambda \times \N$ which is constant in the $\mu$-direction and consider the square
\begin{equation}
	\label{diag:pro-simplicial-2}
\begin{tikzcd}
\{  A_{\lambda} \} \ar[d]\ar[r] & \{ B_{\lambda}  \}\ar[d] \\ 
\{  A_{\lambda} \mmod \bc(\mu)  \}\ar[r] & \{ B_{\lambda} \mmod \bc(\mu)  \}
\end{tikzcd}
\end{equation}
Since $A_{\lambda}$ is noetherian, \cite[Lemma 4.10]{KST} implies that 
\[
\{ \bc(\mu) \pi_{i}(A_{\lambda}) \}_{\mu} \cong \{ \pi_{i}(\fib( A_{\lambda} \to A_{\lambda} \mmod \bc(\mu) ) ) \}_{\mu}
\]
and similarly for $\{B_{\lambda}\}$. Hence the square \eqref{diag:pro-simplicial-2} is weakly cartesian. Furthermore, since 
\[
(A_{\lambda} \mmod \bc(\mu) ) \otimes_{A_{\lambda}} B_{\lambda} 
	\simeq (A_{\lambda} \otimes_{R[x_{1}, \dots, x_{r}]} R) \otimes_{A_{\lambda}} B_{\lambda}
	\simeq B_{\lambda} \otimes_{R[x_{1}, \dots, x_{r}]} R 
	\simeq B_{\lambda} \mmod \bc(\mu),
\]
the second assumption in Theorem~\ref{thm:pro-excision} is trivially satisfied. 
\end{proof}

\begin{rem}\label{rem after cor}
As noted above, Corollaries~\ref{first cor pro exc} and \ref{second cor pro exc} remain valid when replacing $K$-theory with any localizing invariant which is $k$-connective for some $k$ as for example topological Hochschild homology.
\end{rem}

\section{Applications to truncating invariants}
	\label{sec:applications to truncating invariants}

The main goal of this section is to show that truncating invariants satisfy excision and nilinvariance, and to reprove the excision theorems of Cuntz--Quillen, Corti{\~n}as, Geisser--Hesselholt and Dundas--Kittang.
Recall that for a localizing invariant $E$ and an $\E_{1}$-ring $A$, we write $E(A)$ for $E(\Perf(A))$.

\begin{dfn}
	\label{def:truncating-nil-excision}
Let $E \colon \Cat_\infty^\perf \to \cT$ be a localizing invariant. Then $E$ is said to be \emph{truncating} if for every connective $\E_1$-ring spectrum $A$, the canonical map $ E(A) \to E(\pi_0(A)) $
is an equivalence. It is said to be \emph{nilinvariant} if for every nilpotent two-sided ideal $I \subseteq A$ in a discrete unital ring $A$, the canonical map 
$ E(A) \to E(A/I) $
is an equivalence. Finally, $E$ is said to be \emph{excisive}, or to satisfy \emph{excision}, if it sends the  diagram  of $\E_{1}$-ring spectra \eqref{diag:ring-spectra2}
to a pullback square
\[
\begin{tikzcd}
 E(A) \ar[d]\ar[r] & E(B) \ar[d] \\ 
 E(A') \ar[r] & E(B') 
\end{tikzcd}
\]
in $\cT$, provided the square \eqref{diag:ring-spectra2} satisfies the following two conditions:
	\begin{enumerate}
	\item[(E1)] The square \eqref{diag:ring-spectra2} is cartesian and all  $\E_{1}$-rings in it are connective.
	\item[(E2)] The induced map $\pi_{0}(A' \otimes_{A} B) \to \pi_{0}(B')$ is an isomorphism.
	\end{enumerate}
\end{dfn}

\begin{rem}
	\label{rem:more-LT-squares}
Conditions (E1) and (E2) in Definition~\ref{def:truncating-nil-excision} are satisfied for all classes of squares discussed in Example~\ref{ex:LT-squares}, in particular for Milnor squares. Hence an excisive invariant sends a Milnor square to a pullback square. This is what is often called excision classically.

More generally, assume that \eqref{diag:ring-spectra2} is a pullback diagram of connective $\E_1$-ring spectra such that $\pi_0(B) \to \pi_0(B')$ is surjective. Then conditions (E1) and (E2) are satisfied. Indeed,  the connectivity assumption implies that the canonical map 
$  \pi_0(A') \otimes_{\pi_0(A)} \pi_0(B)  \to \pi_0(A'\otimes_A B) $ is an isomorphism. Moreover, as \eqref{diag:ring-spectra2} is a pullback square, and as the map $\pi_{0}(B) \to \pi_{0}(B')$ is surjective,  also the map $\pi_{0}(A) \to \pi_{0}(A')$ is surjective and  the kernel of the latter surjects onto the kernel of the former. This implies
\[
\pi_0(A') \otimes_{\pi_0(A)} \pi_0(B) \cong \pi_{0}(B) / \ker(\pi_{0}(B) \to \pi_{0}(B') ) \cong \pi_{0}(B')
\]
as needed.
\end{rem}

\begin{thm}\label{thm:truncating-excisive}
Any truncating invariant satisfies excision.
\end{thm}
\begin{proof}
Clearly, for any truncating invariant $E$ and any  square of $\E_{1}$-ring spectra \eqref{diag:ring-spectra2} satisfying (E1) and (E2) the map $E(A' \wtimes{A}{B'} B) \to E(B')$ is an equivalence. We thus conclude by \cref{main theorem}.
\end{proof}

\begin{rem}
	\label{rem:k-linear-excision}
If $k$ is some $\E_{\infty}$-ring, and if $E$ is a truncating invariant defined on $k$-linear $\infty$-categories as in Remark~\ref{linear categories}, then $E$ sends  any diagram of connective $k$-algebras \eqref{diag:ring-spectra2} satisfying (E1) and (E2) to a pullback square. We say that $E$ satisfies excision on $k$-algebras.
\end{rem}

\begin{cor}\label{truncating => nil}
Any truncating invariant is nilinvariant. 
\end{cor}
\begin{proof}
The proof is similar to that of Theorem~\ref{thm:torsion-relative-K-groups}.
By induction  we may assume that $I^{2} = 0$. 
Then we  form the connective differential graded algebras $C(I,A)$ and $C(I,A/I)$ as in \eqref{eq:CIA} and the pullback diagram \eqref{diag:CIA-pullback}. Since the map $C(I,A) \to C(I,A/I)$ is a $\pi_{0}$-isomorphism, we may apply  \cref{thm:truncating-excisive} to get the following pullback diagram.
\[
\begin{tikzcd}
	E(A) \ar[r] \ar[d] & E(A/I) \ar[d] \\ 
	E(C(I,A)) \ar[r] & E(C(I,A/I))
\end{tikzcd}
\]
Using that $E$ is truncating and the fact that $C(I,A) \to C(I,A/I)$ is a $\pi_0$-isomorphism again, we find that the lower horizontal map in this pullback is an equivalence. Thus so is the upper horizontal map as claimed.
\end{proof}

From Theorem \ref{thm:truncating-excisive} we obtain simple direct proofs of several  previously known excision results.
Our main new application is the following. 
We denote by $K^\inv$  the fibre of the  cyclotomic trace $K \to \TC$ from $K$-theory to integral topological cyclic homology.

\begin{cor}\label{excision for Kinv}
The fibre    of the  cyclotomic trace $K^\inv$ satisfies excision.
\end{cor}
\begin{proof}
Both $K$-theory and $\TC$ are localizing invariants and the cyclotomic trace is a natural transformation of such, see Corollary 19 in Nikolaus's lecture and Corollary 16 in Gepner's lecture in \cite{OWR-TC}. Thus $K^\inv$ is localizing. The main result of Dundas--Goodwillie--McCarthy  \cite[Theorem 7.0.0.2]{DGM}  implies that $K^\inv$ is truncating.
\end{proof}

\begin{rem}
After profinite completion the same result was proven by Dundas--Kittang \cite{DK1} building on work of Geisser--Hesselholt \cite{GH} in the discrete case. Using also Corti{\~n}as' rational analogue of Corollary~\ref{excision for Kinv} (see Corollary~\ref{cor:Cortinas} below), Dundas and Kittang prove a slightly weaker integral result in \cite{DK2}, namely that $K^{\inv}$ sends the pullback diagram \eqref{diag:ring-spectra2} of connective $\E_{1}$-ring spectra to a pullback if both maps $A' \to B'$ and $B \to B'$ are surjective on $\pi_{0}$. For technical reasons, our result is not obtainable with their method, see \cite[Remark 1.2(3)]{DK2}.
\end{rem}

\begin{ex}
	\label{ex:non-nil}
If $E$ is a truncating invariant, and the ideal $I$ in the discrete ring $A$ is locally nilpotent, i.e.~every element of $I$  is nilpotent, the map $E(A) \to E(A/I)$ need not be an equivalence, as the following example shows.
Let $k$ be a perfect field of characteristic $p >0$, and consider $A = k[x^{p^{-\infty}}]/(x) = \colim_n k[x]/(x^{p^n})$ where the transition maps send $x$ to $x^{p}$. The kernel of the canonical map $A \to k$ is  locally nilpotent, but  not  nilpotent. If $K^\inv(A) \to K^\inv(k)$ were an equivalence,  the $p$-completeness of $\TC(A)$ and $\TC(k)$ would imply that also the relative $K$-theory $K(A,k) \simeq \colim K(k[x]/(x^{p^n}),k)$ is $p$-complete. This is, however, not the case: First we recall from \cite[Theorem A]{HM} that $K_{2}(k[x]/(x^{p^n}),k) = 0$ for all $n$, hence $K_2(A,k) = 0$. We will argue that $\pi_2(K(A,k)^\swedge_p) \neq 0$ which shows that $K(A,k)$ is not $p$-complete. We observe that there is a surjection 
\[ 
\pi_2(K(A,k)^\swedge_p) \lto \lim\limits_{n} \pi_2(K(A,k)/p^n) \cong \lim\limits_n K_{1}(A,k)[p^n] 
\]
where the brackets denote the kernel of multiplication by $p^n$. The last isomorphism comes from the fact that  $K_2(A,k) = 0$. Since the map $A \to k$ has a section, we have 
$K_{1}(A,k)[p^n] \cong \ker( K_1(A)[p^n] \to K_1(k))$. 
The compatible system of units $(1+x^{p^{-n}})_{n\geq 0}$ gives a non-trivial element in $\lim_{n} K_{1}(A,k)[p^n] $,
so that $\pi_2(K(A,k)^\swedge_p) \neq 0$ as claimed.
\end{ex}

Next we want to deduce  Corti{\~n}as' rational analogue of Corollary~\ref{excision for Kinv}, of which it was an important precursor,   from \cref{thm:truncating-excisive}.
To do so, we first have to explain how cyclic homology and its variants fit into our framework.
For this, let $k$ be an $\E_\infty$-ring and consider the category $\Cat_\infty^k$ of $k$-linear  $\infty$-categories as in \cref{linear categories}. As discussed in \cite{Hoyois}, there is a functor 
\[ \HH(-/k) \colon \Cat_\infty^k \lto \Fun(\mathrm{B}\mathbb{T},\Mod_k) \]
sending a $k$-linear  $\infty$-category $\cC$ to its Hochschild homology relative to $k$, equipped with its canonical $\mathbb{T}$-action. 
Following the argument of Keller \cite{Keller} or Blumberg--Mandell \cite[Theorem 7.1]{BM} one  proves that $\HH(-/k)$ takes exact sequences of   $k$-linear  $\infty$-categories to fibre sequences in $\Mod(k)$, so  $\HH(-/k)$ is a localizing invariant of  $k$-linear  $\infty$-categories. It follows that $\HC(-/k)$, $\HN(-/k)$, respectively $\HP(-/k)$ are also localizing invariants for  $k$-linear  $\infty$-categories: They are obtained by applying the exact functors 
of orbits $(-)_{h\T}$, fixed points $(-)^{h\T}$, respectively the Tate construction $(-)^{t\T}$ under the circle group $\T$ to $\HH(-/k)$.
If $k$ is a discrete commutative ring and $\cC = \Perf(A)$ for some $k$-algebra $A$, then these constructions reproduce  Hochschild,  cyclic, respectively   periodic cyclic homology of $A$ relative to $k$, see \cite[Theorem 2.1]{Hoyois}, where we always mean the derived version of these theories.

The composition
\[ 
	\Cat_\infty^\perf \lto \Cat_{\infty}^\Q \xrightarrow{\HN(-/\Q)}  \Mod(\Q)
\]
of extension of scalars from $\mathbb{S}$ to $\Q$ and $\HN(-/\Q)$ is thus a localizing invariant and will be denoted by $\HN_\Q$. We will write $K_{\Q}$ for rational $K$-theory, i.e.\ for the localizing invariant that sends a small, stable $\infty$-category $\cC$ to $K(\cC) \otimes \Q$.

The Goodwillie--Jones Chern character gives a natural transformation $K_{\Q} \to \HN_{\Q}$.  We let $K_\Q^\inf$ be its fibre and call it \emph{rational infinitesimal $K$-theory}. For discrete rings, the following result is due to Corti\~{n}as, see \cite[Main Theorem 0.1]{Cortinas}.

\begin{cor}
	\label{cor:Cortinas}
Rational infinitesimal $K$-theory $K^{\inf}_{\Q}$ satisfies excision.
\end{cor}
\begin{proof}
Since both $K_\Q$ and $\HN(-/\Q)$ are localizing, so is $K^\inf_\Q$. By Goodwillie's theorem \cite[Main Theorem]{Goodwillie}, the map $K_\Q^\inf(R) \to K_\Q^\inf(\pi_{0}(R))$ is an equivalence for every connective $\E_1$-ring $R$ (see Remark~\ref{rem:Goodwillie-for-E1} below). That is, $K_\Q^\inf$ is truncating. We thus conclude by \cref{thm:truncating-excisive}.
\end{proof}

\begin{rem}
	\label{rem:Goodwillie-for-E1}
Note that \cite[Main Theorem]{Goodwillie} is a statement about simplicial rings rather than connective $\E_1$-algebras. The more general case may be reduced to the statement about simplicial rings by applying \cite[Main Theorem]{Goodwillie} to $A\otimes H\Z$ instead of the $\E_1$-ring $A$, see for instance \cite[Corollary 2.2]{AR}.
\end{rem}

For discrete rings and $k$ a field of characteristic zero the following corollary is a  result of Cuntz and Quillen \cite[Theorem 5.3]{CQ}. 
Besides the work of Suslin and Wodzicki \cite{Suslin-Wod,Wodzicki} the ideas of Cuntz and Quillen played a crucial role in the previous approaches to excision results by Corti{\~n}as, Geisser--Hesselholt, and Dundas--Kittang. The extension of the Cuntz--Quillen theorem to arbitrary commutative base rings $k$ containing $\Q$, as we present it here, was also obtained by different techniques by Morrow \cite[Theorem 3.15]{Morrow-pro-unitality}.
\begin{cor}
	\label{cor:Cuntz-Quillen}
Let $k$ be a discrete commutative ring containing $\Q$. Then periodic cyclic homology $\HP(-/k)$ satisfies excision for $k$-algebras.
\end{cor}
\begin{proof}
As explained before, $\HP(-/k)$ is a localizing invariant for $k$-linear  $\infty$-categories. Since $k$ contains the rational numbers, $\HP(-/k)$ is truncating: If $k$ is a field, this is \cite[Theorem IV.2.1]{Goodwillie2}, the more general case follows from \cite[Theorem IV.2.6]{Goodwillie2}.      We conclude  from the $k$-linear version of \cref{thm:truncating-excisive}, see Remark~\ref{rem:k-linear-excision}.
\end{proof}

Finally, we discuss the example of Weibel's homotopy $K$-theory $\KH$. We thank Akhil Mathew and an anonymous referee for encouraging us to include this example in this paper. 
Firstly, recall from Remark~\ref{linear categories} the $\infty$-category $\Cat_\infty^\cE$ of $\cE$-linear $\infty$-categories for $\cE$ a rigid object of $\CAlg(\Cat^{\perf}_{\infty})$. It inherits a symmetric monoidal structure  from $\Cat_\infty^\perf$ given by the relative tensor product $-\otimes_\cE -$. The same is true in the presentable context, and the ind-completion functor refines to a symmetric monoidal functor, see \cite[Section~4]{HoyoisScherotzkeSibilla}. If $\cE$ is given by $\Perf(k)$ for an $\E_\infty$-ring we denote the relative tensor product by $-\otimes_k -$. 

\begin{lemma}
	\label{lemma:tensor-exact}
For every $\cE$-linear $\infty$-category $\cD$, the functor $- \otimes_\cE \cD$ preserves exact sequences.
\end{lemma}
\begin{proof}
As a left adjoint, $-\otimes_{\cE} \cD$ preserves colimits. It then suffices to further argue that $\cA\otimes_{\cE} \cD \to \cB \otimes_\cE \cD$ is fully faithful if $\cA \to \cB$ is. Using that $\Ind$ is symmetric monoidal, it suffices to show that $\Ind(\cA) \otimes_\cE \Ind(\cD) \to \Ind(\cB)\otimes_\cE \Ind(\cD)$ is fully faithful. As $\cE$ is rigid, the functor $\Ind(\cA) \to \Ind(\cB)$ admits an $\cE$-linear right adjoint which preserves colimits \cite[Remark D.7.4.4]{LurieSAG}. Since $-\otimes_{\cE} \Ind(\cD)$ has the structure of a 2-functor \cite[Proof of 3.9]{CMNN}, it preserves the adjunction and the fact that its unit is an equivalence.
\end{proof}

As a localizing invariant for $H\Z$-linear categories, homotopy $K$-theory was introduced by Tabuada \cite{Tabuada}. For convenience, we briefly recall its definition.
For an $\E_{\infty}$-ring $k$, the association $A \mapsto \Perf(A)$ refines to a  symmetric monoidal functor $\Alg(k) \to \Cat^{k}_{\infty}$. In particular $\Perf(A) \otimes_{k} \Perf(B) \simeq \Perf(A \otimes_{k} B)$ for $k$-algebras $A$ and $B$, see \cite[Remark 4.8.5.17]{LurieHA} for the full module categories  and use that $\Ind$ is symmetric monoidal.
We let $\Z[\Delta^\bullet]$ be the  simplicial ring of algebraic simplices, see \cite[Definition IV.11.3]{Weibel} for details.
More generally, for any $H\Z$-algebra $R$ we let $R[\Delta^\bullet] = R\otimes_\Z \Z[\Delta^\bullet]$ and recall that there is an equivalence $\Perf(R[\Delta^\bullet]) \simeq \Perf(R) \otimes_\Z \Perf(\Z[\Delta^\bullet])$ of simplicial objects in $\Cat_\infty^\Z$.

\begin{dfn}
We define \emph{homotopy $K$-theory} as the functor $\Cat^\Z_\infty \to \Sp$ given by the formula
\[
 \KH(\cC) = \colim\limits_{\Delta^\op} K(\cC\otimes_\Z \Perf(\Z[\Delta^\bullet])) .
 \]

\end{dfn}

For a discrete ring $A$ we find 
\[ 
\KH(\Perf(A)) \simeq \colim\limits_{\Delta^\op} K(A[\Delta^\bullet])  = \KH(A)
\]
which is the original definition of Weibel, see \cite[Definition IV.12.1]{Weibel}. 
The following proposition is well-known to the experts, but for sake of completeness we include an argument. 
\begin{prop}\label{KH truncating}
Homotopy $K$-theory
$\KH$ is a truncating invariant of $H\Z$-linear $\infty$-categories.
\end{prop}
\begin{proof}
Using Lemma~\ref{lemma:tensor-exact} the fact that $K$-theory is localizing implies that 
 homotopy $K$-theory is localizing.
To show that it is truncating, we consider a connective $H\Z$-algebra $A$.
We want to prove that the  map $\KH(A) \to \KH(\pi_0(A))$, which is 
the geometric realization of the map of simplicial spectra $K(A[\Delta^\bullet]) \to K(\pi_0(A)[\Delta^\bullet])$, is an equivalence.
Since non-positive $K$-theory of an $\E_{1}$-ring $R$ only depends on $\pi_{0}(R)$ (cf.~Lemma~\ref{lemma:waldhausen-E1}), 
and since $\Omega^{\infty}$ commutes with sifted colimits of connective spectra, it is enough to show that the map on geometric realizations of infinite loop spaces
\[ 
| \BGL(A[\Delta^\bullet])^+| \to |\BGL(\pi_0(A)[\Delta^\bullet])^+| 
\]
is an equivalence.
Since the plus construction commutes with the geometric realization \cite[Lemma 3.1.1.9]{DGM} it  suffices to show that the map $| \BGL(A[\Delta^\bullet])| \to |\BGL(\pi_0(A)[\Delta^\bullet])|$ is an equivalence. We claim that in fact the map 
\[
 | \GL(A[\Delta^\bullet])| \to | \GL(\pi_0(A)[\Delta^\bullet])|
 \]
is an equivalence. Letting $I = \fib(A \to \pi_0(A))$, the fibre of this map is $| \M(I[\Delta^\bullet])|$, where $\M$ denotes the $\E_{\infty}$-space of matrices. This follows for example by applying \cite[Theorem~B.4]{BousfieldFriedlander}.
Now $| \M(I[\Delta^\bullet])|$  is a module over $| \Z[\Delta^\bullet]|$, which is contractible.
Thus the proposition is proven.
\end{proof}

\begin{cor}
Homotopy $K$-theory satisfies excision and nilinvariance.
\end{cor}

For discrete rings, this result is due to Weibel and well known, see \cite[Theorem IV.12.4, Corollary IV.12.5]{Weibel}.

\section{An example}\label{sec:example}

In this section we  discuss the following family of pullback diagrams of rings. Let $k$ be a discrete unital ring, and let $\alpha$ be an element of $k$. Consider the following commutative diagram in which all maps are the canonical ones.
\begin{equation}
	\label{diag:family}
\begin{tikzcd}
	k \ar[r] \ar[d] & k[y] \ar[d] \\ k[x] \ar[r] & k[x,y]/(yx-\alpha)
\end{tikzcd}
\end{equation}
This is a pullback of $\E_1$-rings, and thus \cref{main theorem} provides an $\E_1$-ring $k[x] \wtimes{k}{k[x,y]/(yx-\alpha)} k[y]$ with underlying spectrum $k[x]\otimes_k k[y]$.

\begin{prop}
	\label{ring in example}
The $\E_{1}$-ring $k[x] \wtimes{k}{k[x,y]/(yx-\alpha)} k[y]$  in the above example is discrete and isomorphic to 
\[ 
k\langle x,y \rangle/{(yx-\alpha)},
\] 
the non-commutative polynomial ring over $k$ in two non-commuting variables $x$ and $y$ modulo the relation $yx=\alpha$. 
\end{prop}

\begin{proof}
Since the map $k \to k[x]$ is flat, the $\E_{1}$-ring $k[x] \wtimes{k}{k[x,y]/(yx-\alpha)} k[y]$ is discrete with underlying $k$-module $k[x] \otimes_{k} k[y]$.
From Propositions~\ref{prop:identification-boxtimes-B} and \ref{prop:identification-boxtimes-A} we know that under this identification the ring maps $k[x] \to k[x] \otimes_{k} k[y]$ and $k[y] \to k[x] \otimes_{k} k[y]$ as well as the $k[x]$-left module structure and the $k[y]$-right module structure are the canonical ones.
From the formula
\[
(x^{k}\otimes y^{l}) \cdot (x^{m}\otimes y^{n}) = (x^{k} \otimes 1) \cdot (1\otimes y^{l}) \cdot (x^{m}\otimes 1) \cdot (1\otimes y^{n})
\]
we see that it remains to describe the $k[y]$-left module structure on $k[x] \otimes_{k} k[y]$. 
Let $I$ be the fibre of the right vertical map in \eqref{diag:family}.  
By Proposition~\ref{prop:identification-boxtimes-B} the $k[y]$-left module structure on $k[x]\otimes_{k} k[y]$ is determined by the cofibre sequence of $k[y]$-left modules
\[
I \otimes_{k} k[y] \lto k[y] \lto k[x] \otimes_{k} k[y].
\]
We represent $\Sigma I$ by the complex $[k[y] \to k[x,y]/(yx-\alpha)]$ concentrated in degrees 1 and 0. Then the rotation of the previous cofibre sequence is represented by the following diagram of complexes of  left $k[y]$-modules: 
\[
\begin{tikzcd}
\begin{bmatrix} 0 \\ \downarrow \\ k[y] \end{bmatrix}  \ar[r] & \begin{bmatrix} 0 \\ \downarrow \\ k[x]\otimes_{k} k[y] \end{bmatrix} \ar[r, "j"] & 
	\begin{bmatrix}  k[y]\otimes_k k[y] \\ \downarrow \\ k[x,y]/(yx-\alpha)\otimes_k k[y] \end{bmatrix}
\end{tikzcd}
\]
Since the map $j$ is injective and $k[y]$-left linear, we may calculate the left-multiplication by $y$ 
after applying the map $j$. For any $m \geq 1$ we get
\[
j((1\otimes y) \cdot (x^{m} \otimes 1)) = (yx^{m}) \otimes 1 =   j(\alpha\cdot x^{m-1} \otimes 1).
\]
This finishes the proof.
\end{proof}

From the previous proposition and Theorem \ref{main theorem} we can compute $E(k\langle x,y \rangle/(yx-\alpha))$ for any localizing invariant $E$. We  denote by $NE(k)$ the cofibre of the canonical map $E(k) \to E(k[x])$, so that $E(k[x]) \simeq E(k) \oplus NE(k)$. We get:
\begin{cor}\label{calculation of Toeplitz ring}
	\label{cor:K-theory-nc-laurent}
There is a canonical equivalence
\[ 
E(k\langle x,y \rangle/(yx-\alpha)) \simeq E(k) \oplus NE(k) \oplus NE(k).
\]
\end{cor}
Recall that one can use the semi-orthogonal decomposition of $\Perf(\mathbb{P}^1_k)$ generated by $\mathcal{O}$ and $\mathcal{O}(-1)$ and the standard covering of $\mathbb{P}^1_k$ to obtain the fundamental theorem calculating $E(k[x,x^{-1}])$ for any localizing invariant $E$.
Using Corollary~\ref{calculation of Toeplitz ring} one can give an alternative  proof of the fundamental theorem 
parallel to Cuntz's proof of Bott periodicity \cite[Section 4]{Cuntz}. For homotopy $K$-theory this was done by Corti\~nas--Thom \cite[Section 7.3]{CortinasThom}. 
Namely, let us consider the case where $\alpha = 1$, so that $k[x,y]/(yx-\alpha) \cong k[x, x^{-1}]$ is the ring of  Laurent polynomials. We observe that the ring $k\langle x,y\rangle/(yx-1) = T_k$ is  the Toeplitz ring, which is also used in \cite{CortinasThom}.
It sits inside the Toeplitz extension
\[
\begin{tikzcd}
	0 \ar[r] & \M(k) \ar[r] & T_k \ar[r] & k[x,x^{-1}] \ar[r] & 0
\end{tikzcd}
\]
where $\M(k) = \colim_{n} \M_{n}(k)$ is the ring of finite matrices  and the map $\M(k) \to T_{k}$ sends the basic matrix $e_{ij}$ to the element $x^{{i-1}}(1-xy)y^{{j-1}}$. 
As a filtered colimit of unital rings, $\M(k)$ is Tor-unital, i.e.~the projection $\Z \ltimes \M(k) \to \Z$ is Tor-unital, and hence $\M(k)$ satisfies excision for any localizing invariant by Corollary~\ref{cor intro}. Moreover, by classical Morita theory
the non-unital ring homomorphism $k \to \M(k)$ that is the inclusion in the upper left corner induces an equivalence of perfect modules over the unitalizations
 $\Perf(\Z \ltimes k) \xrightarrow{\sim} \Perf(\Z \ltimes \M(k))$. Hence we get a fibre sequence
\begin{equation}
	\label{seq:Toeplitz-extension}
\begin{tikzcd}
	E(k) \ar[r] & E(T_k) \ar[r] & E(k[x,x^{-1}]).
\end{tikzcd}
\end{equation}
We claim that the first map in \eqref{seq:Toeplitz-extension} is nullhomotopic, so that $E(k[x,x^{-1}]) \simeq E(T_k) \oplus \Sigma E(k)$. To see the claim, let $e \in T_{k}$ be the idempotent $1-xy$. Then $eT_{k}$ is a $k$-$T_{k}$-bimodule, which is perfect as a $T_{k}$-right module,
and the map in question is induced by the functor $\Perf(k) \to \Perf(T_{k})$ that sends $P$ to $P \otimes_{k} eT_{k}$.
There is a $k$-$T_{k}$-bimodule isomorphism $T_{k} \oplus eT_{k} \cong T_{k}$ given by $(m,en) \mapsto xm+en$ with inverse $m \mapsto (ym, em)$.
From this we deduce a natural equivalence 
\[
(-) \otimes_{k} T_{k} \oplus (-) \otimes_{k} eT_{k}  \simeq (-) \otimes_{k} T_{k} 
\]
of functors $\Perf(k) \to \Perf(T_{k})$. By additivity, this implies that $(-) \otimes_{k} eT_{k}$ induces the zero map on any localizing invariant.
Using Corollary~\ref{calculation of Toeplitz ring} and the fibre sequence \eqref{seq:Toeplitz-extension} we obtain the fundamental theorem:

\begin{cor} 	
\label{cor:fundamental-theorem}
Let $k$ be a  discrete unital ring, and let $E$ be a localizing invariant. There is an equivalence
\[ 
E(k[x,x^{-1}]) \simeq \Sigma E(k) \oplus E(k) \oplus NE(k) \oplus NE(k).
\]
\end{cor}

\appendix

\section{cdh-descent for truncating invariants}
	\label{sec:cdh}

In this appendix we prove that any truncating invariant satisfies cdh-descent.
We fix a quasi-compact and quasi-separated (qcqs for short) base scheme $S$, and we denote by $\Sch_{S}$ the category of quasi-separated $S$-schemes of finite type. By an abstract blow-up square we mean a pullback square 
\begin{equation}
	\label{diag:abs}
\begin{tikzcd}
 D \ar[d]\ar[r] & \widetilde X \ar[d,"p"] \\ 
 Y \ar[r,"i"] & X 
\end{tikzcd}
\end{equation}
in $\Sch_{S}$ where $i$ is a finitely presented closed immersion, and $p$ is finitely presented, proper, and an isomorphism over $X \setminus Y$.
A blow-up square is a square \eqref{diag:abs} in which $\widetilde X = \Bl_{Y}(X)$. A blow-up square is thus an abstract blow-up square if and only if $\Bl_{Y}(X) \to X$ is finitely presented, e.g.~if $X$ is noetherian. We call such squares finitely presented blow-up squares.
The cdh-topology on $\Sch_{S}$ is generated by the Nisnevich coverings and the families $\{ Y \to X, \widetilde X \to X\}$ for any abstract blow-up square as above.

As it will become relevant shortly, let us recall the notion of a derived blow-up square from \cite[\S 3.1]{KST}.  Given a finite sequence $\vec{a} = (a_{1}, \dots, a_{r})$ of elements of the commutative ring $A$, we choose a commutative noetherian\footnote{The assumption that $A'$ be noetherian is not essential.  It was used in \cite{KST}  to prove that the derived blow-up is independent of choices.} ring $A'$ together with a regular sequence $\vec{a}' = (a'_{1}, \dots, a'_{r})$ and a ring map $A' \to A$ sending $\vec{a}'$ to $\vec{a}$.
We set $X = \Spec(A)$, $Y = V((\vec{a})) \subseteq X$, and $X'$, $Y'$ analogously
and consider the following blow-up square.
\begin{equation}
	\label{diag:regular-bu}
\begin{tikzcd}
 D' \ar[d]\ar[r] & \Bl_{Y'}(X') \ar[d] \\ 
 Y' \ar[r] & X' 
\end{tikzcd}
\end{equation}
The derived pullback of the square \eqref{diag:regular-bu} along the map $X \to X'$ is a square of derived schemes for which we use the notation
\begin{equation}
	\label{diag:derived-bu}
\begin{tikzcd}
 \D \ar[d]\ar[r] & \widetilde{\X} \ar[d] \\ 
 \Y \ar[r] & X.
\end{tikzcd}
\end{equation}
Such a square is called a derived blow-up square associated to the sequence of elements $\vec{a}$. For any derived scheme $\X$, we denote by $t\X$ its underlying scheme. Then $t\Y \cong Y$, the square $t\eqref{diag:derived-bu}$ is cartesian, and $t\widetilde{\X}$ contains the ordinary blow-up $\Bl_{Y}(X)$ as a closed subscheme.

If $E$ is a localizing invariant, and if $X$ is a  qcqs (derived) scheme, we define $E(X)$ to be $E(\Perf(X))$ where $\Perf(X)$ denotes the $\infty$-category of perfect $\cO_{X}$-modules.
We  call a commutative square of schemes $E$-cartesian if  $E$ sends  it to a pullback square. Similarly, we call a morphism of schemes an $E$-equivalence if  $E$ sends it  to an equivalence.
For convenience we record the following  statements as a lemma.

\begin{lemma}\label{lemma appendix}
\begin{enumerate}
\item Every localizing invariant satisfies Nisnevich descent.
\item A localizing invariant $E$ satisfies cdh descent if and only if every abstract blow-up square is $E$-cartesian.
\item For every localizing invariant $E$, a derived blow-up square is $E$-cartesian. 
\end{enumerate}

\end{lemma}
\begin{proof}
By a theorem of Voevodsky, a presheaf $E$ on $\Sch_S$ satisfies Nisnevich, respectively cdh-descent if and only if elementary Nisnevich squares, respectively elementary Nisnevich squares and abstract blow-up squares are $E$-cartesian, see e.g.\ \cite[Theorem 3.2.5]{AHW}. Using this, (i) follows from Thomason--Trobaugh, \cite{ThomasonTrobaugh} and (ii) from (i). For (iii), one notes that the proof of \cite[Theorem 3.7]{KST} applies for any localizing invariant in place of $K$-theory.
\end{proof}

The proof of the next theorem follows the arguments  of \cite[Proof of Theorem A]{KST}. 

\begin{thm}\label{cdh}
Let $E$ be a truncating invariant. Then $E$ satisfies cdh-descent on $\Sch_{S}$. Equivalently, it sends any abstract blow-up square \eqref{diag:abs} of qcqs schemes to a pullback square.
\end{thm}
In Appendix~\ref{app:B} we extend Theorem~\ref{cdh} to truncating invariants with coefficients in a quasi-coherent sheaf of connective algebras on $S$.

\begin{proof} 
We show that any abstract blow-up square \eqref{diag:abs} is $E$-cartesian.

\begin{step1}
If the map $p$ in diagram \eqref{diag:abs} is a closed immersion, then \eqref{diag:abs} is $E$-cartesian.
\end{step1}

By Zariski descent we may assume that $X$ and hence all schemes   in \eqref{diag:abs} are affine, say $X = \Spec(A)$, $Y = \Spec(A/I)$, 
 $\widetilde{X} = \Spec(A/J)$ with finitely generated ideals $I$ and $J$, and thus  $D = \Spec(A/(I+J))$.
As $E$  satisfies excision by \cref{thm:truncating-excisive}, the Milnor square
\[ 
\begin{tikzcd}
	A/(I\cap J) \ar[r] \ar[d] & A/I \ar[d] \\
	A/J \ar[r] & A/(I+J)
\end{tikzcd}
\]
is $E$-cartesian. Since the map $A \to A/J$ is an isomorphism outside $Y$ and $I$ and $J$ are finitely generated, there is an integer $N >0$ such that $I^{N}\cdot J = 0$. Hence $I \cap J$ is nilpotent, and the map $A \to A/(I\cap J)$ is an $E$-equivalence  by \cref{truncating => nil}. This implies the claim.

\begin{step2}
If \eqref{diag:abs} is a pullback of a finitely presented blow-up square, then it is $E$-cartesian.
\end{step2}

We may again assume that $X$ is affine. Suppose first that \eqref{diag:abs} is a finitely presented blow-up square.
Pick a finite sequence of generators of the ideal defining $Y \subseteq X$ and form an associated derived blow-up square as in \eqref{diag:derived-bu}.
We obtain the following commutative diagram:
\[\begin{tikzcd}
	D \ar[r] \ar[d] & t\D \ar[r] \ar[d] & \D \ar[r] \ar[d] & \Y \ar[d] & Y \ar[l] \ar[dl] \\
	\widetilde{X} \ar[r] & t\widetilde{\X} \ar[r] & \widetilde{\X} \ar[r] & X 
\end{tikzcd}\]
The left square is $E$-cartesian by Step~1, the middle square is $E$-cartesian because $E$ is truncating, and the right square is a derived blow-up square and thus $E$-cartesian by \cref{lemma appendix}. Finally, the map $Y \to \Y$ is an $E$-equivalence, again since $E$ is truncating. Thus the original square \eqref{diag:abs} is $E$-cartesian. The same argument applies to any pullback of \eqref{diag:abs} along some map $X' \to X$ as by definition the formation of the underlying scheme of a derived pullback  commutes with base change.

\begin{step3}
Assume that \eqref{diag:abs} is the pullback of an abstract blow-up square
\begin{equation}
	\label{diag:abs0}
\begin{tikzcd}
 D_{0} \ar[d]\ar[r] & \widetilde X_{0} \ar[d] \\ 
 Y_{0} \ar[r] & X_{0} 
\end{tikzcd}
\end{equation}
of noetherian schemes along a map $f\colon X \to X_{0}$ where $\widetilde X_{0} =\Bl_{Y'_{0}}(X_{0})$ for some closed subscheme $Y_{0}'$ of $Y_{0}$. Then \eqref{diag:abs} is $E$-cartesian.
\end{step3}
Consider the diagram of pullback squares:
\[\begin{tikzcd}
	D'' \ar[r] \ar[d] \ar[dr, phantom, "(1)"] & f^{*}\Bl_{Y'_{0}}(Y_{0}) \ar[d] & \\
	D' \ar[r] \ar[d] \ar[dr, phantom, "(2)"]& D \ar[r] \ar[d] \ar[dr, phantom, "(3)"]& \widetilde{X} \ar[d] \\
	f^{*}Y'_{0} \ar[r] & Y \ar[r] & X
\end{tikzcd}\]
Square (1) is $E$-cartesian  by Step 1. The square obtained by combining (1) and (2) is $E$-cartesian by Step 2, as is the square obtained by combining (2) and (3). Thus so is square~(3).

\begin{step4}
Any abstract blow-up square  \eqref{diag:abs} is $E$-cartesian.
\end{step4}

By noetherian approximation \cite[Appendix C]{ThomasonTrobaugh} we can write \eqref{diag:abs} as the pullback of an abstract blow-up square \eqref{diag:abs0}
of noetherian schemes along a map $f\colon X \to X_{0}$.
Let $\hat X_{0}$ by the scheme theoretic closure of $X_{0} \setminus Y_{0}$ in $X_{0}$ and $\hat Y_{0}= Y_{0} \times_{X_{0}} \hat X_{0}$.
Then $\{ Y \to X, f^{*}\hat X_{0} \to X\}$ is a finitely presented closed cover of $X$, and hence the map of relative terms $E(X,Y) \to E(f^{*}\hat X_{0}, f^{*}\hat Y_{0})$ is an equivalence by Step 1. Similarly, if $\hat{\widetilde{X}}_{0}$ denotes the scheme theoretic closure of $X_{0} \setminus Y_{0}$ in $\widetilde X_{0}$, and $\hat D_{0}$ the pullback, then $E(\widetilde X, D) \simeq E( f^{*}\hat{\widetilde{X}}_{0}, f^{*}\hat D_{0})$. So by replacing \eqref{diag:abs0} by its hatted version we may assume that $X_{0} \setminus Y_{0}$ is scheme-theoretically dense in $X_{0}$ and $\widetilde X_{0}$.
In this situation one can use Raynaud--Gruson's \emph{platification par \'eclatements} to find a closed subscheme $Y_{0}' \to X_{0}$ which is set-theoretically contained in $Y_{0}$ and such that there is a factorization $\Bl_{Y_{0}'}(X_{0}) \to \widetilde{X_{0}} \to X_{0}$, see \cite[Lemma 2.1.5]{Temkin}.
Since $E$ is nilinvariant by \cref{truncating => nil}, we can replace $Y_{0}$ by a nilpotent thickening of $Y_{0}$ (and $Y$ by the pullback) so that $Y'_{0}$ becomes a closed subscheme of $Y_{0}$ without changing the value $E(Y)$.
We then consider the following diagram of pullback squares.
\[\begin{tikzcd}
	D' \ar[r] \ar[d] \ar[dr, phantom, "(1)"] & D \ar[r] \ar[d] \ar[dr, phantom, "(2)"] & Y \ar[d] \\
	f^{*}\Bl_{Y'_{{0}}}(X_{{0}}) \ar[r] & \widetilde{X} \ar[r] & X
\end{tikzcd}\]
The composite of relative $E$-terms
\[ 
E(X,Y) \stackrel{\alpha}{\lto} E(\widetilde{X},D) \stackrel{\beta}{\lto} E(f^{*}\Bl_{Y'_{0}}(X_{0}),D') 
\]
is an equivalence by Step 3. Thus, we have now proven that for any abstract blow-up square, the induced map $\alpha$ of relative $E$-terms has a left-inverse. Since  diagram $(1)$ is itself an abstract blow-up square, it follows that $\beta$ itself has a left-inverse, and thus is an equivalence. Thus also $\alpha$ is an equivalence as needed.
\end{proof}

\begin{rem}
Example~\ref{ex:non-nil} shows that one cannot in general drop the assumption that $i$ is finitely presented. However, if the truncating invariant $E$ commutes with filtered colimits, then one may drop the finite presentation assumption on  $i$ by writing $i$ as a cofiltered limit of finitely presented closed immersions. 
In fact, one can also drop the finite presentation assumption on $p$: Step 1 then works for $i$ and $p$ not necessarily finitely presented closed immersions, Step 2 for arbitrary blow-up squares, Step 3 works for $\widetilde X \to X$ given by the blow-up of an arbitrary closed subscheme of $Y$. Finally, for Step 4 we may assume that $Y \subseteq X$ is finitely presented and its complement is scheme-theoretically dense in $X$ and $\widetilde X$. One can then directly apply \cite[Lemma 2.1.5]{Temkin}. 
Examples of truncating invariants which commute with filtered colimits are  homotopy $K$-theory and $K^{\inv}/n$, see \cite[Theorem G]{CMM}.
\end{rem}

\begin{cor}\label{cdhforfibre}
The fibre  of the cyclotomic trace $K^\inv$ satisfies cdh-descent.
\end{cor}
For schemes essentially of finite type over an infinite perfect field admitting strong resolution of singularities, a similar result was proven  by Geisser and Hesselholt \cite[Theorem B]{GH3} using an argument of Haesemeyer \cite{Haesemeyer}.

\begin{cor}
 Homotopy $K$-theory satisfies cdh-descent.
\end{cor}
\begin{proof}
This follows from Theorem~\ref{cdh} and Proposition~\ref{KH truncating}.
\end{proof}

For  schemes over $\Q$ cdh-descent of $KH$ was first proven by Haesemeyer \cite{Haesemeyer}. The general case was proven later by Cisinski \cite{Cisinski} using proper base change in motivic homotopy theory.  
In \cite{KST} it was then shown that homotopy $K$-theory is in fact the cdh-sheafification of $K$-theory if the base scheme is noetherian. Recently, Khan  \cite{Khan} gave a proof of cdh-descent for homotopy $K$-theory using a similar approach.

The following result was proven in \cite{MR2415380} using Haesemeyer's approach.
\begin{cor}
The fibre  of the rational Goodwillie--Jones Chern character $K^{\inf}_{\Q}$ and periodic cyclic homology $\HP(-/k)$ over commutative $\Q$-algebras $k$ satisfy cdh-descent.
\end{cor}

The following is the analog of \cite[Theorem A]{KST} for topological cyclic homology.

\begin{cor}
	\label{cor:pro-cdh-TC}
Topological cyclic homology  satisfies `pro-descent' for abstract blow up squares of noetherian schemes, i.e.~for any abstract blow-up square of noetherian schemes \eqref{diag:abs}, the square of pro-spectra
\[
\begin{tikzcd}
 \TC(X) \ar[d]\ar[r] & \TC(\widetilde X) \ar[d] \\ 
 \{ \TC(Y_{n}) \} \ar[r] & \{ \TC( D_{n} ) \} 
\end{tikzcd}
\]
where $Y_{n}$ and $D_{n}$ denote the $n$-th infinitesimal thickening of $Y$ in $X$ and $D$ in $\widetilde X$, respectively, is weakly cartesian.
\end{cor}
\begin{proof}
This follows directly from Corollary~\ref{cdhforfibre} and the corresponding statement for $K$-theory which is \cite[Theorem A]{KST}. 
\end{proof}

One checks that the only arguments in the proof of \cite[Theorem A]{KST} that are seemingly specific to $K$-theory are the following two:
\begin{enumerate}
\item The fact that $K$-theory satisfies Nisnevich descent and descent for derived blow-up squares. This holds for any localizing invariant by \cref{lemma appendix}.
\item The validity of \cite[Theorem 4.11]{KST}, which is precisely \cref{second cor pro exc}. This  holds for any $k$-connective localizing invariant by \cref{rem after cor}.
\end{enumerate}
We thus obtain the following result.

\begin{thm}\label{pro cdh descent for THH}
If $E$ is a localizing invariant which which is $k$-connective for some $k$,
then $E$ satisfies pro-descent for abstract blow-up squares of noetherian schemes. In particular, topological Hochschild homology $\THH$ satisfies pro-descent for abstract blow-up squares of noetherian schemes.
\end{thm}

Previously, Morrow proved that $\THH$ and $\TC$ with finite coefficients satisfy pro-descent for abstract blow-up squares of noetherian, $F$-finite $\Z_{(p)}$-schemes of finite dimension, see \cite[Theorem 3.5]{Mor}.

\section{cdh-descent for truncating invariants with coefficients}
\label{app:B}

Let $\cE \in \CAlg(\Cat_\infty^\mathrm{perf})$ be rigid, and let $E$ be a localizing invariant for $\cE$-linear $\infty$-categories. For any $\cE$-linear $\infty$-category $\cZ$, one can define a variant of $E$ with coefficients in $\cZ$ by the formula
\[ 
\cC \mapsto E(\cC;\cZ) = E(\cC\otimes_\cE \cZ).
\]
A similar construction was considered by Tabuada \cite{MR3178252}.
Since $-\otimes_\cE \cZ$ preserves exact sequences by Lemma~\ref{lemma:tensor-exact}, $E(-;\cZ)$ is again localizing. Furthermore, if $E$ is truncating, $k$ is a connective $\E_{\infty}$-ring, and $R$ is a connective $k$-algebra, then $E(-;\Perf(R))$ is also truncating: For every $k$-algebra $A$ we have $E(A;\Perf(R)) \simeq E(A\otimes_k R)$. If in addition $A$ is connective, there is a commutative diagram
\[
\begin{tikzcd}[column sep=tiny]
	E(A\otimes_k R) \ar[rr] \ar[dr,"\simeq"'] & & E(\pi_0(A)\otimes_k R) \ar[dl,"\simeq"] \\
	& E(\pi_0(A\otimes_k R))
\end{tikzcd}
\]
in which both diagonal maps are equivalences since $E$ is truncating. 

We now apply these observations to prove cdh-descent for truncating invariants with coefficients.
For any qcqs scheme $X$ we write $\QCoh(X)$ for the presentably symmetric monoidal, stable $\infty$-category of quasi-coherent $\cO_{X}$-modules, see \cite[Section~2.2.4]{LurieSAG}.
Note that $\QCoh(X)$ is compactly generated with compact objects $\QCoh(X)^{\omega} = \Perf(X)$, see \cite[Proposition~9.6.1.1]{LurieSAG}.
For an algebra $\cR \in \Alg(\QCoh(X))$ we define the presentable, stable $\infty$-category of quasi-coherent $\cR$-modules
\[
\QCoh(\cR) = \RMod_{\cR}( \QCoh(X) )
\]
as the $\infty$-category of  $\cR$-right modules in $\QCoh(X)$.
The free objects $P \otimes_{\cO_{X}} \cR$ with $P \in \Perf(X)$ form a set of compact generators of $\QCoh(\cR)$ and we define 
\[
\Perf(\cR) = \QCoh(\cR)^{\omega},
\] 
which is a $\Perf(X)$-linear category.

\begin{ex}
If $X = \Spec(A)$ is affine, then $\cR$ is equivalently given by an $HA$-algebra $R$, and $\QCoh(\cR) = \RMod_{R}( \Mod(A) ) \simeq \RMod(R)$ by \cite[Theorem 7.1.3.1]{LurieHA} and likewise $\Perf(\cR) \simeq \Perf(R)$.
\end{ex}

We now fix a qcqs base scheme $S$ and write $\Sch_{S}$ for the category of qcqs $S$-schemes of finite type. We further fix  an algebra $\cR \in \Alg(\QCoh(S))$. For $X \in \Sch_{S}$, we denote the pullback of $\cR$ to $X$ by $\cR_{X}$. 
\begin{lemma}
There are natural equivalences $\QCoh(\cR_{X}) \simeq \QCoh(X) \otimes_{\QCoh(S)} \QCoh(\cR)$ and $\Perf(\cR_{X}) \simeq \Perf(X) \otimes_{\Perf(S)} \Perf(\cR)$.
\end{lemma}
\begin{proof}
Using the definition and  \cite[Theorem 4.8.4.6]{LurieHA} we have
\begin{align*}
\QCoh(X) \otimes_{\QCoh(S)} \QCoh(\cR) &= \QCoh(X) \otimes_{\QCoh(S)} \RMod_{\cR}(\QCoh(S)) \\
	&\simeq \RMod_{\cR}( \QCoh(X) ) \\
	&\simeq \RMod_{\cR_{X}}(\QCoh(X))\\
	&= \QCoh(\cR_{X}).
\end{align*}
The second equivalence follows by passing to compact objects.
\end{proof}

Let $E$ be a localizing invariant of $\Perf(S)$-linear $\infty$-categories. 
For ease of notation we write  $E(-;\cR)$ for the localizing invariant $E(-;\Perf(\cR))$.
Note that for $X$ in $\Sch_{S}$, we have $E(X;\cR) \simeq E(\Perf(\cR_{X}))$ and $E(-;\cR)$ satisfies Nisnevich descent on $\Sch_{S}$  by Lemma~\ref{lemma appendix}.

\begin{thm}
Let $E$ be a truncating invariant of $\Perf(S)$-linear $\infty$-categories, and assume that $\cR$ is connective. Then $E(-;\cR)$ satisfies cdh-descent on $\Sch_{S}$.
\end{thm}
\begin{proof}
By Zariski descent we may assume that $S=\Spec(A)$ is affine. In that case $\cR$ is given by a connective $HA$-algebra $R$, and  $E(-; \cR)$ is truncating by the considerations in the beginning of this appendix. It thus satisfies cdh-descent by Theorem~\ref{cdh}.
\end{proof}

\bibliographystyle{amsalpha}
\bibliography{pullbacks}

\end{document}